\documentclass[reqno,10pt]{amsart}
\usepackage{amssymb,amsmath,amsthm,}
\usepackage{a4wide}
\usepackage{amscd}
\usepackage{amsfonts}
\usepackage{amssymb}
\usepackage{latexsym}
\usepackage{color}
\usepackage{esint}
\usepackage{graphicx}
\usepackage{caption}
\usepackage{subcaption}
\usepackage{amsthm}
\usepackage{verbatim}
\usepackage{hyperref}

\usepackage[makeroom]{cancel}
\usepackage{mathtools}

\let\hide\iffalse

\newtheorem{theorem}{Theorem}[section]

\newtheorem{lemma}[theorem]{Lemma}

\newtheorem{remark}[theorem]{Remark}

\renewcommand{\S}{\mathbb{S}}

\newcommand{\be}{\begin{equation}}
\newcommand{\bm}{\begin{multline}}
\newcommand{\ee}{\end{equation}}

\newcommand{\Bes}{\begin{eqnarray*}}
	\newcommand{\Ees}{\end{eqnarray*}}
\newcommand{\Be}{\begin{equation}}
\newcommand{\Ee}{\end{equation}}

\def\B{\begin{equation}}
\def\E{\end{equation}}
\def\BN{\begin{eqnarray*}}
\def\EN{\end{eqnarray*}}

\numberwithin{equation}{section}

\usepackage{tikz-cd}
\usepackage{ragged2e}
\usepackage{etoolbox}
\usepackage{lipsum}

\usepackage{color}

\title[] {High-velocity tails of the inelastic and the multi-species mixture Boltzmann equations}
\author[] {Gayoung An and Donghyun Lee}

	\address{Department of Mathematics, POSTECH, Pohang-si, Gyeongsangbuk-do, 37673 Republic of Korea
		\\agy19@postech.ac.kr}
	\address{Department of Mathematics, POSTECH, Pohang-si, Gyeongsangbuk-do, 37673 Republic of Korea
	\\donglee@postech.ac.kr}

\begin{document}
\begin{abstract}
    We study high-velocity tails of some homogeneous Boltzmann equations on $v \in \mathbb{R}_{v}^d$. First, we consider spatially homogeneous {\it Inelastic} Boltzmann equation with {\it noncutoff} collision kernel, in the case of moderately soft potentials. We also study spatially homogeneous mixture Boltzmann equations : for both noncutoff collision kernel with moderately soft potentials and cutoff collision kernel with hard potentials. In the case of noncutoff {\it Inelastic} Boltzmann, we obtain 
    \[
    	f(t,v) \geq a(t) e^{-b(t)|v|^p}, \quad 2 < p < 6.213
    \]
 	by extending Cancellation lemma \cite{AD2000} and spreading lemma \cite{IM2020} and assuming $f\in C^{\infty}$. For the Mixture type Boltzmann equations, we prove Maxwellian $p=2$.
\end{abstract}

\maketitle
\tableofcontents

\section{Introduction}
The classical Boltzmann equation describes the statistical behavior of a rarefied collisional gas, consisting of a very large number of identical
particles and perfect elastic binary collisions. The function $f(t,x,v)$ is a density of particles at time $t$, at point $x$, having speed $v$. The equation is a partial differential equation of $f(t,x,v)$ with nonlinear quadratic term $Q(f,f)(t,x,v).$ In this paper, we study two different types of the Boltzmann equation : Inelastic Boltzmann equation (no energy conservation) and the multi-species Mixture Boltzmann. \\

\noindent($\mathbf{I}$) {\bf The Inelastic(mono-species) Boltzmann equation model} : 
Let $v$ and $v_*$ be the velocities of two particles before the collision, and $v'$ and $v'_*$ be the velocities after the collision. The particles hit each other with the unit vector ${n}\in \S^{d-1}_{+}$ and assume that
\begin{align} \label{alpha}
\langle v'-v_*',n \rangle &= -\alpha \langle v-v_*,n \rangle, \notag \\
(v'-v'_*)-\langle v'-v_*',n \rangle {n} &= (v-v_*)-\langle v-v_*,n \rangle {n},
\end{align}
where $\alpha$ is the coefficient of normal restitution, $0<\alpha<1$. In particular, we denote 
\Be \label{def beta}
	\beta = \frac{1+\alpha}{2},\quad \frac{1}{2} < \beta < 1. 
\Ee 
Postcollisional velocities $v',v'_*$ can be expressed as
\begin{equation} \label{v'_I}
\begin{split}
    v' &= v-\beta\langle v-v_*,n \rangle {n}
    =\frac{v+v_*}{2}+\frac{1-\beta}{2}(v-v_*)+\frac{\beta}{2}\;|v-v_*|{\sigma},  \\
     v_*' &= v_*+\beta\langle v-v_*,n \rangle {n}=\frac{v+v_*}{2}-\frac{1-\beta}{2}(v-v_*)-\frac{\beta}{2}\;|v-v_*|{\sigma}
\end{split}
\end{equation}
for the angular parameter ${n} \in \mathbb{S}^{d-1}_{+}, \; {\sigma} \in \mathbb{S}^{d-1}.$
It is important to note that unlike to elastic collision, there is no symmetry between $0 < \frac{v-v_*}{|v-v_*|} \cdot \sigma < \frac{\pi}{2}$ and $\frac{\pi}{2} < \frac{v-v_*}{|v-v_*|} \cdot \sigma < \pi$. i.e., taking minus on ${\sigma}$ does not reverse $v'$ and $v_*'$($v' \nleftrightarrow v_*'$ when ${\sigma} \leftrightarrow -{\sigma}$). Also, the collision is not time-reversible and $|v'-v'_*| \neq|v-v_*|$.
\vspace{3mm}

\noindent The spatially homogeneous Inelastic Boltzmann equation is described by 
\begin{align} \label{BT_I}
    \partial_t f (t,v) = Q(f,f) (t,v),
\end{align} 
where $v \in \mathbb{R}^d, \; t \in \mathbb{R}_{+}.$ The function $f(t,v)$ is a density of particles of velocity $v$ and time $t$. The Inelastic collision operator $Q(f,f) (t,v)$ in \eqref{BT_I} satisfies mass and momentum conservation, but not for energy. 
%So, we suppose the initial values of mass and energy in 
We write 
\begin{align} \label{hydro_I}
    \int f(t,v) dv = \int f_0(v) dv \doteq M_0,\quad \int f(t,v)|v|^2 dv \leq \int f_0(v)|v|^2 dv \doteq E_0
\end{align} 
for some finite initial mass and energy $M_0, E_0$. \vspace{3mm}

\noindent Next, we introduce the weak form of the Inelastic Boltzmann collision operator $Q(f,f)(t,v)$,
\begin{align} \label{Q_nc_I}
     \int Q(f,g)(v) \phi(v) \; dv &= \int B(|v-v_*|, \cos \theta)f(v_*)g(v)(\phi(v')-\phi(v))\: d\sigma dv_*dv, \notag \\
     \quad &\text{where }\quad \cos\theta = \langle\frac{v-v_*}{|v-v_*|}, \sigma\rangle, \quad \theta \in [0, \pi].
\end{align}
Here, we take a suitable regular test function $\phi(v)$ (which has proper decaying property at infinity) and $v'$ is given by \eqref{v'_I}. We define $B(|v-v_*|, \cos \theta)$ as
\begin{equation} \label{B_nc_I}
    B(|v-v_*|,\cos \theta)=|v-v_*|^{\gamma}b(\cos\theta), \quad b(\cos\theta) \approx_{\theta \sim 0} |\theta|^{-(d-1)-2s} \tilde{b}(\cos\theta), \quad
    \int  \tilde{b}(\cos\theta) \;d\sigma < +\infty
\end{equation} with $\gamma>-d,$\; $s\in(0,1).$ The $\tilde{b}(\cos\theta)$ is smooth on $[0,\pi]$ and positive on $[0,\pi).$
When $d \geq 2$, the following condition  %$\gamma<0$ and $\gamma+2s \in [0,2]$ is called the moderately soft potentials. Rewrite, 
\begin{align} \label{soft_I}
    \gamma<0, \quad \gamma+2s \in [0,2] \quad \text{for} \quad d\geq 2
\end{align}
 is called the moderately soft potentials in \cite{IM2020, V2002}. \vspace{7mm}

\noindent ($\mathbf{M}$) {\bf The (elastic) multi-species Mixture Boltzmann equation model} : There are $N$ different types particles in the system and each particle has mass $m_{j}$ for $j=1,2 \dots N.$ Suppose that a particle with mass $m_i$ and velocity $v$ collides with another particle with mass $m_j$ and velocity $v_*$.  We use $v^{\prime}$ and $v_{*}^{\prime}$ to denote postcollisinoal velocities. The laws of conservation of momentum and energy can be stated as follows :
\begin{align} \label{coll_M}
\begin{split}
     m_i v + m_j v_* &=  m_i v^{\prime} + m_j v_{*}^{\prime} \\
    m_i |v|^2 + m_j |v_{*}|^2 &=  m_i |v^{\prime}|^2 + m_j |v_{*}^{\prime}|^2. 
\end{split}
\end{align}
Let ${n}$ denote collision angle. The postcollisional velocities $v^{\prime},v^{\prime}_{*}$ can be written by 
\begin{align} \label{v'_M}
\begin{split}
    v' &=v-\frac{2m_j}{m_i+m_j}\langle v-v_*, n \rangle {n}= \frac{m_i}{m_i+m_j}v+\frac{m_j}{m_i+m_j}v_*+\frac{m_j}{m_i+m_j}|v-v_*|\sigma, \\
    v_*' &= v_*+\frac{2m_i}{m_i+m_j}\langle v-v_*, n \rangle {n}=\frac{m_i}{m_i+m_j}v+\frac{m_j}{m_i+m_j}v_*-\frac{m_i}{m_i+m_j}|v-v_*|\sigma
\end{split}
\end{align}
for angular parameter ${n} \in \mathbb{S}^{d-1}_{+}, \; {\sigma} \in \mathbb{S}^{d-1}.$ The collision is time-reversible and $|v'-v'_*| = |v-v_*|$. However, the collision is not symmetric when $0 < \frac{v-v_*}{|v-v_*|} \cdot \sigma < \frac{\pi}{2}$ and $\frac{\pi}{2} < \frac{v-v_*}{|v-v_*|} \cdot \sigma < \pi$ ($v' \nleftrightarrow v_*'$ when ${\sigma} \leftrightarrow -{\sigma}$) unlike the elastic collision between two identical particles.  \\

\noindent The spatially homogeneous (elastic) Mixture Boltzmann equation is described by 
\begin{align} \label{BT_M}
    \partial_t f_i (t,v) = \sum_{j=1}^{N} Q_{ji}(f_j,f_i) (t,v), 
\end{align} where $v \in \mathbb{R}^d, \; t \in \mathbb{R}_{+}.$ The $f_i(t,v)$ is a density of particles of mass $m_i$ at time $t$ and velocity $v$ for $1 \leq i \leq N$. And we assume $m_1<m_2 \cdots < m_N$ WLOG. If $i \neq j$, $Q_{ji}(f_j,f_i)(t,v)$ is the Mixture Boltzmann collision operator for $1 \leq j \leq N.$ If $i=j$, $Q_{ii}(f_i,f_i)(t,v)=Q(f_{i}, f_{i})(t,v)$ is standard Boltzmann collision operator. The Mixture collision operator in \eqref{BT_M} satisfies mass, momentum, and energy conservations. Summing initial values of mass, energy, and entropy of density functions for $1 \leq i \leq N$, we define 
\begin{align} \label{hydro_M}
\sum_{j=1}^{N} \int f_j(0,v) dv \doteq M_0, \quad
\sum_{j=1}^{N} \int f_j(0,v)|v|^{2} dv \doteq E_0\quad
\sum_{j=1}^{N} \int f_j(0,v)\log f_j(0,v) dv \doteq H_0
\end{align} 
for some finite values $M_0$, $E_0$, and $H_0$. \\

$(\lowercase\expandafter{\romannumeral1})$ With noncutoff collision kernel, the Mixture collision operator $Q_{ji}(f_j, f_i) (t,v)$ is written by  
\begin{align} \label{Q_nc_M}
    Q_{ji}(f_j, f_i) (t,v) &= \int B_{ji}(|v-v_*|, \cos \theta)(f_j(v_*')f_i(v')-f_j(v_*)f_i(v)) \; d\sigma dv_*, \notag \\ \quad &\text{where} \quad \cos\theta = \langle\frac{v-v_*}{|v-v_*|}, \sigma\rangle, \quad \theta \in [0, \pi].
\end{align} Here, $v'$ and $v_*'$ are given by \eqref{v'_M}.
We define $B_{ji}(|v-v_*|,\cos \theta )$ as
\begin{equation} \label{B_nc_M}
    B_{ji}(|v-v_*|,\cos \theta)=|v-v_*|^{\gamma}b_{ji}(\cos\theta), \quad b_{ji}(\cos\theta) \approx_{\theta \sim 0} |\theta|^{-(d-1)-2s} \tilde{b}_{ji}(\cos\theta), \quad
    \int  \tilde{b}_{ji}(\cos\theta) \;d\sigma < +\infty
\end{equation} with $\gamma>-d,\;s\in(0,1).$ The $\tilde{b}_{ji}(\cos\theta)$ is smooth on $[0,\pi]$ and positive on $[0,\pi).$
The values of $d, \gamma$ and $s$ satisfy the moderately soft potentials condition, \eqref{soft_I}. \vspace{3mm}

$(\lowercase\expandafter{\romannumeral2})$ With cutoff collision kernel, for $v \in \mathbb{R}^3$, the Mixture Boltzmann collision operator $Q_{ji}(f_j,f_i)(t,v)$ is written by 
\begin{align} \label{Q_c_M}
   Q_{ji}(f_j, f_i) (t,v) &= \int B_{ji}(|v-v_*|, \theta)(f_j(v_*')f_i(v')-f_j(v_*)f_i(v)) \; dn dv_*, \notag \\ \quad &\text{where} \quad \cos\theta = \langle\frac{v-v_*}{|v-v_*|}, n\rangle, \quad \theta \in [0, \frac{\pi}{2}].
\end{align} Here, $v'$ and $v_*'$ are given by \eqref{v'_M}. We define the $B_{ji}(|v-v_*|, \theta)$ as
\begin{align}  \label{B_c_M}
 B_{ji}(|v-v_*|,\theta) = h_{ji}(\theta)|v-v_*|^\gamma ,\quad \int_0^{\pi/2}h_{ji}(\theta) \;d\theta < +\infty 
\end{align} for hard potentials, $\gamma \in [0,1].$ \vspace{7mm}

In 1932, Carleman  proved a lowerbound for the spatially homogeneous Boltzmann equation for hard potentials with cutoff in dimension 3 at first in \cite{CT1932}. In 1997, Pulvirenti and Wennberg proved that the form of the lowerbound is exactly a Maxwellian in \cite{AB1996}. The lowerbound is uniform on time when $t>t_0$ for any positive time $t_0$ and depends on initial mass, energy, and entropy. In 2005, Mouhot extended the result to the full Boltzmann equation in the torus, $(x,v) \in \mathbb{T}_{x}^{N}\times \mathbb{R}_{v}^{N}$ in \cite{M2005}. So, he proved the lowerbound that the exponential power of $v$ is $2+\epsilon$ for small $\epsilon$ without cutoff. In 2020, Imbert, Mouhot, and Silvestre proved the Gaussian lowerbounds for the Boltzmann equation in the torus without cutoff under only controlling the natural local hydrodynamic quantities in  \cite{IM2020}. In addition, Imbert and Silvestre obtained $C^{\infty}$ estimates for the inhomogeneous Boltzmann equation without cutoff in \cite{IS2021}, 2021. Before, Desvillettes and Villani proved the solutions converging to equilibrium under two assumptions that the solution stays $C^{\infty}$ and is bounded below by some fixed Maxwellian in \cite{LC2005}, 2005. Therefore, we derive the solutions which converge to equilibrium under controlling natural local hydrodynamic quantities.  \\

We briefly introduce some of the studies on the Inelastic Boltzmann equation. In 2004, Gamba, Panferov, and Villani studied the spatially homogeneous Inelastic Boltzmann equation for hard spheres with diffusive term. They proved existence, smoothness and uniqueness of the solution and gave pointwise lowerbound estimates in \cite{IV2004}. Bobylev, Gamba, and Panferov studied the model with zero external forcing term or three types of nonzero external forcing term. They proved the exponential tail of order of steady velocity distribution range from 1 to 2 in \cite{AIV2004}. In 2006, Mischler, Mouhot, and Ricard developed the Cauchy theory with zero external forcing and proved that the solutions converge to the Dirac mass in weak* measure sense when $t \rightarrow \infty$($=$cooling process) in \cite{MMR2006}. Next, Mischler, and Mouhot proved the existence and uniqueness of the self-similar solution and time asymptotic convergence of the solution toward the self-similar solution in \cite{SC2006} and \cite{SC2009}. On the other hand, in 2016, Briant, and Daus studied the Cauchy theory and proved exponential trend to equilibrium for the homogeneous a multi-species mixture Boltzmann equation in \cite{ME2016}. Alonso and Orf studied a \textit{priori} estimates for long range interactions for hard potentials in 2022, \cite{AO2022}. \\

In this paper, we study the spatially homogeneous Inelastic Boltzmann equation without cutoff. In preliminaries, we consider the points $P$ and $Q$ which satisfy  $P-v\;\bot\;P-v_*$ and $v'-Q\;\bot\; v'-v$, respectively, when $v,v_*,v'$, and $v_*'$ are given in \eqref{v'_I}. (See Figure \ref{pre_I_F}.) We split the collision operator into singular and nonsingular parts in \eqref{split_nc_I}. The singular parts, $Q^{s}(f,f)$ changes into Carleman alternative representation form and is expressed by non symmetric function, $K_f(u,u')$ in \eqref{Kf_nc_I}. For test function $\psi(v)$, we estimate on $Q^{s}(f,\psi)(v)$  under the condition of moderately soft potentials. We prove that the Cancellation lemma of nonsingular parts, $Q^{ns}(f,f)$ and $Q^{ns}(f,f)$ is also positive and well-defined. We change the form of the Inelastic collision operator into \eqref{pos_nc_2_I} and apply the maximum principle (e.g. p.103, Villani's note, \cite{V2002}), then we prove that $f(t,v)$ is strictly positive under the condition, $f \in C^{\infty}$. Using the geometric relation of points $v,v_*$, and the point $Q$, we estimate the integration of the region of $v$ and $v_*$ when $v'$ is fixed. Lastly, using \cite{IM2020} and aforementioned lemmas, we prove the spreading lemma and find the lowerbound of the inelastic model. \\

Next, we study the spatially homogeneous Boltzmann equation for multi-species elastic particles, i.e., mixture, without cutoff. 
We split the collision operator $Q_{ji}(f_j,f_i)$ in \eqref{BT_M} into singular parts $Q^{s}_{ji}(f_j,f_i)$ and nonsingular parts $Q^{ns}_{ji}(f_j,f_i)$, and estimate on $Q^{s}_{ji}(f_j,\psi)$. We prove the Cancellation lemma of $Q^{ns}_{ji}(f_j,f_i)$ and $Q^{ns}_{ji}(f_j,f_i)$ is also positive and well-defined. Similar to inelastic model, we prove that $f_i(t,v)$ is strictly positive under the condition of $f_i(t,v) \in C^{\infty}.$ Since $Q^{ns}_{ji}(f_j,f_i)$ is positive and RHS in \eqref{BT_M} includes the general collision operator $Q_{ii}(f_i,f_i)$,
RHS is greater than $Q^{s}_{ii}(f_i,f_i)$ term. In cutoff mixture model, using a similar argument as above, we can retain the gain term of $Q_{ii}(f_i,f_i)$ with time parts in loweround. There is the Gaussian lowerbound for (elastic mono-species) general Boltzmann equation.
It has been proved in \cite{IM2020} for noncutoff collision kernel and in \cite{AB1996} for cutoff collision kernel. Lastly, we can easily get the Gaussian lowerbound in mixture model by applying the paper, \cite{IM2020} and \cite{AB1996}.  \\

The spatially homogeneous Inelastic Boltzmann equation with cutoff collision kernel was already studied by Mischler and Mouhot. In \cite{SC2006} and \cite{SC2009}, they proved that the rescaled solution of the Inelastic Boltzmann equation for hard spheres has lowerbound. Alonso and Orf proved the Cancellation lemma for homogeneous mixture Boltzmann equation when $\theta \in [0, \frac{\pi}{2}]$ in \cite{AO2022} and we extend the range of $\theta$ to $[0,\pi]$ and prove the $Q^{ns}(f_j,f_i)(v)$ is positive and well-defined.  \\

\begin{theorem} (Noncutoff, inelastic, mono-species) \label{1.1}
 Let $f(t,v)$ be the solution of the inelastic Boltzmann equation in \eqref{BT_I}. The collision kernel satisfies noncutoff condition, \eqref{Q_nc_I}, \eqref{B_nc_I}. Assume that $\gamma<0$ and $\gamma+2s \in [0,2]$(moderately soft potentials) and $f \in C^{\infty}$.
Then for any positive time, there are some functions,  $a(t)>0$ and $b(t)>0$ depending on $d,s,M_0$ and $E_0$ in \eqref{hydro_I} such that 
\begin{align} \label{p_I}
    f(t,v) \geq a(t) e^{-b(t)|v|^p}, \quad p =\frac{\log2}{\log\sqrt{1+\beta^2}}, 
\end{align} where $\beta$ is given in \eqref{def beta}\;and\; $2< p < \frac{\log2}{\log\sqrt{5}-\log 2} \simeq  6.213.$ 
\end{theorem} 

\begin{theorem} (Noncutoff, elastic, multi-species) \label{1.2}
%The $f_k(t,v)$ is density of particles of mass $m_k$ for $1\leq k \leq N$ and 
Let $f_i(t,v)$ be the solution of the mixture Boltzmann equation in \eqref{BT_M} and assume that $m_1<m_2< \cdots <m_N$. The collision kernel satisfies noncut off condition, \eqref{Q_nc_M}, \eqref{B_nc_M}. Assume that $\gamma<0$ and $\gamma+2s \in [0,2]$ (moderately soft potentials) and $f_i(t,v) \in C^{\infty}$. Then for any positive time, there are some functions, $a_i(t)>0$ and $b_i(t)>0$ depending on $d,s,M_0$ and $E_0$ in \eqref{hydro_M} such that 
\begin{align} \label{p_M}
    f_i(t,v) \geq a_i(t) e^{-b_i(t)|v|^{2}}. 
\end{align} 
\end{theorem} 

\begin{remark}
The $p$ of \eqref{p_I} is increasing if $\beta$ is decreasing, so $p$ is greater than $2$. Also, total entropy $\sum_{j=1}^{N} \int f_j(t,v)\log f_j(t,v)\; dv$ in mixture model decreases with time and converges steady-state when all $\{f_i(t,v)\}_{i=1}^{N}$ have the form of Gaussian. So, it is natural they have Gaussian lowerbound in Theorem \ref{1.2}. \\
\end{remark} 

\begin{theorem} (Cutoff, elastic multi-species) \label{1.3}
Let $f_i(t,v)$ be the solution of the mixture Boltzmann equation in \eqref{BT_M} and assume that $m_1<m_2< \cdots <m_N$. The collision kernel satisfies cutoff condition, \eqref{Q_c_M}, \eqref{B_c_M}(hard potentials). Then for any positive time, there are some functions $a_i(t)>0$ and $b_i(t)>0$ depending on $M_0,E_0,H_0$ and $\gamma$ in \eqref{hydro_M} such that \eqref{p_M} holds. Moreover, $a_i(t)$ and $b_i(t)$ can be chosen uniformly for all $t_0<t$, where $t_0$ is any positive time. \\
\end{theorem}

\begin{remark}
The $\{f_i(t,v)\}_{i=1}^{N}$ which each $f_i(t,v)$ is the solution in \eqref{BT_M} have the uniformly Gaussian lowerbound for all $t_0<t$, where $t_0$ is any positive time.  \\
\end{remark}  

\section{Preliminaries}

\noindent($\mathbf{I}$) The Inelastic Boltzmann equation model : The $v'$ and $v'_*$ can be expressed as the formula
\begin{align} \label{v'_I'}
\begin{split}
     v'&=v-\beta\langle v-v_*,n \rangle {n}
    =\frac{v+v_*}{2}+\frac{1-\beta}{2}(v-v_*)+\frac{\beta}{2}\;|v-v_*|{\sigma}, \\
     v_*'&=v_*+\beta\langle v-v_*,n \rangle {n}=\frac{v+v_*}{2}-\frac{1-\beta}{2}(v-v_*)-\frac{\beta}{2}\;|v-v_*|{\sigma}
\end{split}
\end{align}
for the angular parameter ${n} \in \mathbb{S}^{d-1}_{+}, \; {\sigma} \in \mathbb{S}^{d-1}.$ The geometry of the inelastic collision defined by \eqref{v'_I'} is shown in Figure \ref{pre_I_F}. 
Let the point $O$ be $\frac{v+v_*}{2}$. Also, let the point $A$ and $B$ be $\frac{v+v_*}{2}+\frac{1-\beta}{2}(v-v_*)$ and $\frac{v+v_*}{2}-\frac{1-\beta}{2}(v-v_*)$, respectively. 
\begin{figure}[t]
\centering
\includegraphics[width=6cm]{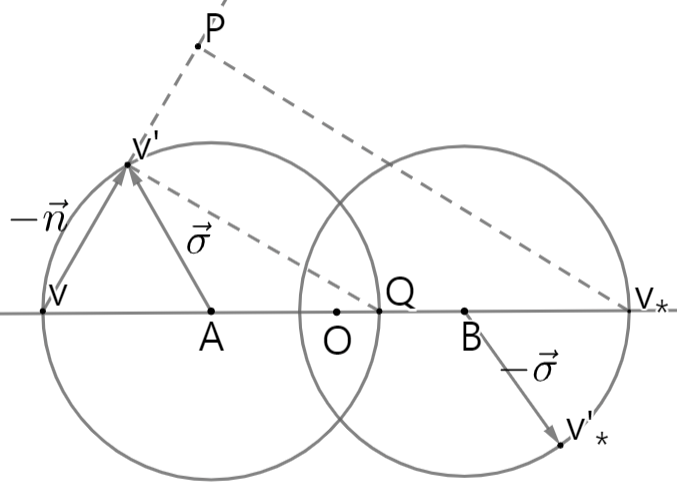} 
\caption{Given $v$ and $v_*$, the possible locations of $v'$, $v_*'$ when $\frac{1}{2}<\beta < \frac{2}{3}.$ \label{pre_I_F}} 
\end{figure}
The point $P,$ which is located on the extension line of $v$ and $v'$ in ratio $|P-v|:|P-v'|=1:1-\beta$, is as follows :
\begin{align} 
    P &= \frac{1}{\beta}v'-(\frac{1}{\beta}-1)v, \quad P-v\;\bot\;P-v_*.  \label{P_I}  
\end{align}
The point $Q,$ which internally divides the line segment joining the points $v$ and $v'$ in the ratio $|Q-v|:|Q-v_*|=\beta:1-\beta$, is as follows : 
\begin{align}
    Q &= (1-\beta)v+\beta v_*, \quad v'-Q\;\bot\; v'-v. \label{Q_I}
\end{align}
When  $v$ and $v'$ are fixed, we define the plane $E_{Pv'}$ with a normal vector $v'-v$ that contains the point $P$ as 
\begin{align} \label{E_pv'_I}
    E_{Pv'}\doteq \{x \in \mathbb{R}^d|\; (x-P)\;\bot \;v-v' , \quad P=\frac{1}{\beta}v'-(\frac{1}{\beta}-1)v\;\}.
\end{align} 
We express $\cos \frac{\theta}{2}$ and $\sin \frac{\theta}{2}$ as 
\begin{align} 
    \cos \frac{\theta}{2} &= \frac{|Q-v'|}{|Q-v|} =
    \frac{|(1-\beta)v+\beta v_*-v'|}{\beta|v-v_*|}, \label{cos_I}  \\
    \sin \frac{\theta}{2} &= \frac{|v-v'|}{|Q-v|}=
    \frac{|v-v'|}{\beta|v-v_*|}, \label{sin_I}
\end{align} where $\langle \frac{v-v_*}{|v-v_*|}, \sigma \rangle = \cos\theta$.
From \eqref{v'_I'}, 
\begin{align} \label{sigma_I}
     &{\sigma}=\frac{v'-(\frac{v+v_*}{2}+\frac{1-\beta}{2}(v-v_*))}{\frac{\beta|v-v_*|}{2}}=\frac{v'-v+\frac{\beta}{2}(v-v_*)}{\frac{\beta|v-v_*|}{2}}.
\end{align} Therefore, 
\begin{align} \label{delta_I}
     &\delta(|\sigma|^2-1)
    = \delta(\frac{(v-v')\cdot(\beta v_*+(1-\beta)v-v')}{\frac{\beta^2|v-v_*|^2}{4}})
    =\frac{\beta|v-v_*|^2}{4|v-v'|}\delta(\frac{v-v'}{|v-v'|}\cdot (v_*+(\frac{1}{\beta}-1)v-\frac{1}{\beta}v')).
\end{align} \vspace{7mm}

\noindent ($\mathbf{M}$) The Mixture Boltzmann equation model : The $v',v_*'$ can be expressed as the formula
\begin{align} \label{v'_M'}
\begin{split}
     v' =v-\frac{2m_j}{m_i+m_j}\langle v-v_*, n \rangle {n}= \frac{m_i}{m_i+m_j}v+\frac{m_j}{m_i+m_j}v_*+\frac{m_j}{m_i+m_j}|v-v_*|
    \sigma,\\
    v_*' =v_*+\frac{2m_i}{m_i+m_j}\langle v-v_*, n \rangle {n}=\frac{m_i}{m_i+m_j}v+\frac{m_j}{m_i+m_j}v_*-\frac{m_i}{m_i+m_j}|v-v_*|\sigma
\end{split}
\end{align}
for the angular parameter ${n} \in \mathbb{S}^{d-1}_{+}, \; {\sigma} \in \mathbb{S}^{d-1}.$ The geometry of the mixture collision defined by \eqref{v'_M'} is shown in Figure \ref{Pre_M_F_1} and Figure \ref{Pre_M_F_2}. Let the point $O$ be $\frac{m_i}{m_i+m_j}v+\frac{m_j}{m_i+m_j}v_*$.\vspace{3mm}

\noindent First, in the case $m_i<m_j$, the geometry of the mixture collision defined by \eqref{v'_M'} is shown in Figure \ref{Pre_M_F_1} and the radius of $v'$ is greater than the radius of  $v_*'$. The point $P,$ which internally divides the line segment joining the points $v$ and $v'$ in the ratio $|P-v| : |P-v'|=m_j-m_i : m_i+m_j$, is as follows :
\begin{equation} \label{Pre_P_M_1}
    P=\frac{m_i+m_j}{2m_j}v+\frac{m_j-m_i}{2m_j}v',
    \quad P-v'\;\bot\;P-v_*'.
\end{equation}
The point $Q,$ which is located on the extension line of $v'$ and $v_*$ in ratio $|Q-v'| : |Q-v_*'|=2m_j : m_j-m_i$, is as follows :
\begin{equation} \label{Pre_Q_M_1}
    Q=\frac{2m_j}{m_i+m_j}v_*'-\frac{m_j-m_i}{m_i+m_j}v', \quad v-Q\;\bot\;v-v'.
\end{equation}
\begin{figure}[t]
\centering
\includegraphics[width=4cm]{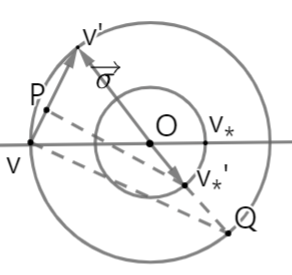} 
\caption{Given $v$ and $v_*$, the possible locations of $v'$, $v_*'$ \label{Pre_M_F_1}}
\end{figure} 
When $v'$ and $v$ are fixed, we define the plane $E_{Pv'}$ with normal vector $v'-v$ that contains the point $P$ as 
\begin{align} \label{E_pv'_M_1}
    E_{Pv'}\doteq \{x \in \mathbb{R}^d|\; (x-P)\;\bot \;v-v' , \quad P=\frac{m_i+m_j}{2m_j}v+\frac{m_j-m_i}{2m_j}v'\;\}.
\end{align} 
We express $\cos \frac{\theta}{2}$ and $\sin \frac{\theta}{2}$ as 
\begin{align} 
    \cos \frac{\theta}{2} &= \frac{|v-Q|}{|v'-Q|}=
    \frac{|v-\frac{2m_j}{m_i+m_j}v_*'+\frac{m_j-m_i}{m_i+m_j}v'|}{\frac{2m_j}{m_i+m_j}|v-v_*|}, \label{cos_M_1} \\
    \sin \frac{\theta}{2} &= \frac{|v'-v|}{|v'-Q|}=
    \frac{|v'-v|}{\frac{2m_j}{m_i+m_j}|v-v_*|}, \label{sin_M_1}
\end{align} where $\langle \frac{v-v_*}{|v-v_*|}, \sigma \rangle = \cos\theta.$
From \eqref{coll_M} and \eqref{v'_M'},
\begin{align} \label{sigma_M_1}
      &{\sigma}=\frac{v'-(\frac{2m_j}{m_i+m_j}v_*'-\frac{m_j-m_i}{m_i+m_j}v')}{\frac{2m_j}{m_i+m_j}|v-v_*|}=\frac{v'-v-(\frac{2m_j}{m_i+m_j}v_*'-\frac{m_j-m_i}{m_i+m_j}v'-v)}{\frac{2m_j}{m_i+m_j}|v-v_*|}.
\end{align}
Therefore,
\begin{align} \label{delta_M_1}
\begin{split}
     \delta(|\sigma|^2-1)
    &= \delta(\frac{2(v-v')\cdot(\frac{2m_j}{m_i+m_j}v_*'-\frac{m_j-m_i}{m_i+m_j}v'-v)}{(\frac{2m_j}{m_i+m_j})^2|v-v_*|^2}) \\ 
    &= \frac{m_j}{m_i+m_j}\frac{|v-v_*|^2}{|v-v'|}\delta(\frac{v-v'}{|v-v'|}\cdot (v_*'-\frac{m_j-m_i}{2m_j}v'-\frac{m_i+m_j}{2m_j}v)).
\end{split}
\end{align} 
\vspace{3mm}

\noindent Second, in the case $m_i>m_j$, the geometry of the mixture collision defined by \eqref{v'_M'} is shown in Figure \ref{Pre_M_F_2} and the radius of $v'$ is smaller than the radius of  ${v_*}'$. The point $R,$ which is located on the extension line of $v'$ and $v$ in the ratio $|R-v'| : |R-v|=m_i-m_j : m_i+m_j$, is as follows :
\begin{equation} \label{Pre_R_M_2}
    R=\frac{m_i+m_j}{2m_j}v'-\frac{m_i-m_j}{2m_j}v, \quad R-v\;\bot\;R-v_*.
\end{equation}
The point $S,$ which internally divides the line segment joining the points $v$ and $v'$ in the ratio $|S-v| : |S-v_*|=2m_j : m_i-m_j$, is as follows :
\begin{equation} \label{Pre_S_M_2}
    S=\frac{m_i-m_j}{m_i+m_j}v+\frac{2m_j}{m_i+m_j}v_*,
    \quad S-v'\;\bot\;v-v'.
\end{equation}
\begin{figure} [t]
\includegraphics[width=4cm]{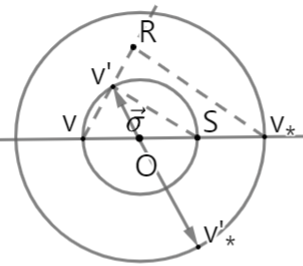} 
\caption{Given $v$ and $v_*$, the possible locations of $v'$, $v_*'$\label{Pre_M_F_2}} 
\end{figure}

\noindent When $v'$ and $v$ are fixed, we define the place $E_{Rv'}$ with normal vector $v'-v$ that contains the point $R$ as 
\begin{align} \label{E_Rv'_M_2}
    E_{Rv'}\doteq \{x \in \mathbb{R}^d|\; (x-R)\;\bot \;v-v' , \quad R=\frac{m_i+m_j}{2m_j}v'-\frac{m_i-m_j}{2m_j}v\;\}.
\end{align} 
We express $\cos \frac{\theta}{2}$ and $\sin \frac{\theta}{2}$ as 
\begin{align} 
    \cos \frac{\theta}{2} &= \frac{|S-v'|}{|S-v|} =
    \frac{|(1-\frac{2m_j}{m_i+m_j})v+\frac{2m_j}{m_i+m_j} v_*-v'|}{\frac{2m_j}{m_i+m_j}|v-v_*|},  \label{cos_M'}  \\
    \sin \frac{\theta}{2} &= \frac{|v-v'|}{|S-v|}=
    \frac{|v-v'|}{\frac{2m_j}{m_i+m_j}|v-v_*|}, \label{sin_M'}
\end{align} where $\langle \frac{v-v_*}{|v-v_*|}, \sigma \rangle = \cos\theta$.
From \eqref{v'_M'}, 
\begin{align} \label{sigma_M'}
     &{\sigma}=\frac{v'-(\frac{m_i}{m_i+m_j}v+\frac{m_j}{m_i+m_j}v_*)}{\frac{m_j}{m_i+m_j}|v-v_*|}=\frac{v'-v+\frac{m_j}{m_i+m_j}(v-v_*)}{\frac{m_j}{m_i+m_j}|v-v_*|}.
\end{align}
Therefore,
\begin{align} \label{delta_M'}
\begin{split}
      \delta(|\sigma|^2-1)
    &= \delta(\frac{(v-v')\cdot(\frac{2m_j}{m_i+m_j} v_*+(1-\frac{2m_j}{m_i+m_j})v-v')}{(\frac{m_j}{m_i+m_j})^2|v-v_*|^2}) \\
    &=\frac{m_j}{2(m_i+m_j)}\frac{|v-v_*|^2}{|v-v'|}\delta(\frac{v-v'}{|v-v'|}\cdot (v_*+\frac{m_i-m_j}{2m_j}v-\frac{m_i+m_j}{2m_j}v')).
\end{split}
\end{align} \\

\section{The Inelastic Boltzmann equation}
\subsection{Noncutoff collision kernel}
Taking $\phi(v)$ to be a suitably regular test function, the weak form of the Inelastic Boltzmann collision operator $Q(f,f)(t,v)$ is written by 
\begin{align} \label{Q_nc_I'}
\begin{split}
        \int Q(f,g)(v) \phi(v) \; dv &= \int B(|v-v_*|, \cos \theta)f(v_*)g(v)(\phi(v')-\phi(v))\: d\sigma dv_*dv,\\
     \quad &\text{where}\;\;\cos\theta = \langle\frac{v-v_*}{|v-v_*|}, \sigma\rangle, \quad \theta \in [0, \pi]
\end{split}
\end{align}
and $v'$ is given by in \eqref{v'_I}. We define the $B(|v-v_*|, \cos \theta)$ as   
\begin{equation} \label{B_nc_I'}
    B(|v-v_*|,\cos \theta)=|v-v_*|^{\gamma}b(\cos\theta), \quad b(\cos\theta) \approx_{\theta \sim 0} |\theta|^{-(d-1)-2s} \tilde{b}(\cos\theta), \quad
    \int  \tilde{b}(\cos\theta) \;d\sigma < +\infty.
\end{equation}
The $\tilde{b}(\cos\theta)$ is smooth on $[0,\pi]$ and positive on $[0,\pi).$ Because taking minus on $\sigma$ does not reverse $v'$ and ${v_*'}$($v' \nleftrightarrow v_*'$ when $\sigma \leftrightarrow -\sigma$), we can not reduce the range of $\theta$, $[0, \pi]$ to $[0, \frac{\pi}{2}].$ By loss of energy, the collision is not time-reversible, i.e., $B(|v-v_*|,\cos\theta) \neq B(|v'-v_*'|,\cos\theta).$  \vspace{3mm}

\noindent We split \eqref{Q_nc_I'} into two parts,
\begin{align} \label{split_nc_I}
\begin{split}
     \int Q(f,g)(u)\phi(u) \;du 
    &= \int B(|u-v_*|, \cos\theta)g(u)f(v_*) (\phi(u')-\phi(u)) \;d\sigma dv_* du \\
     &= \int B(|u-v_*|, \cos\theta)f(v_*)(\phi(u')g(u)-\phi(u')g(u'))\; d\sigma dv_*du \\
      &\quad + \int B(|u-v_*|, \cos\theta)f(v_*)(\phi(u')g(u')-\phi(u)g(u))\; d\sigma dv_* du. 
\end{split}
\end{align}
After taking $\phi(u)=\delta(u-v)$, we define the singular part $Q^{s}(f,g)(v)$ and the nonsingular part $ Q^{ns}(f,g)(v)$ as
\begin{align}  
    Q^{s}(f,g)(v) &= \int B(|u-v_*|, \cos\theta)f(v_*)\delta(u'-v)(g(u)-g(u'))\; d\sigma dv_*du, \label{Qs_nc_I} \\
    Q^{ns}(f,g)(v) &= \int B(|u-v_*|, \cos\theta)f(v_*)(\delta(u'-v)g(u')-\delta(u-v)g(u))\; d\sigma dv_* du \label{Qns_nc_I}
\end{align}
and rewrite $Q(f,g)(v)$ as
\begin{align} \label{Qs_Qns_nc_I}
    Q(f,g)(v) = Q^{s}(f,g)(v) + Q^{ns}(f,g)(v).
\end{align} \vspace{3mm}

\noindent First, we change the $Q^{s}(f,g)(v)$ into the Carleman alternative representation form.
We rewrite $u$ instead of $v$ and $u'$ instead of $v'$ in \eqref{v'_I'}, then 
\begin{align} \label{u'_I}
\begin{split}
    u'=u-\beta\langle u-v_*,n \rangle {n}
    =\frac{u+v_*}{2}+\frac{1-\beta}{2}(u-v_*)+\frac{\beta}{2}\;|u-v_*|{\sigma}, \\
     v_*'=v_*+\beta\langle u-v_*,n \rangle {n}=\frac{u+v_*}{2}-\frac{1-\beta}{2}(u-v_*)-\frac{\beta}{2}\;|u-v_*|{\sigma},
\end{split}
\end{align} where $\cos \theta = \langle \frac{u-v_*}{|u-v_*|}, \sigma \rangle.$ By changing of variables $\sigma \rightarrow u'$ with Jacobian determinant $|\frac{d u'}{d\sigma}|= (\frac{\beta|u-v_*|}{2})^{d}$ and replacing $\delta(|\sigma|^2-1)$ by \eqref{delta_I}, we have that
\begin{align} 
    &\quad\int B(|u-v_*|,\cos\theta)f(v_*)\delta(u'-v)(g(u)-g(u')) \; d\sigma du dv_* \notag \\
     &=\int\int (\frac{\beta|u-v_*|}{2})^{-d}|u-v_*|^{\gamma}b(\cos\theta)f(v_*)\delta(u'-v)(g(u)-g(u'))\delta(|\sigma|^2-1) \; du'dudv_* \notag \\
    &=2^{d-2}\beta^{-d+1}\int |u'-u|^{-1}\delta(u'-v)(g(u)-g(u'))\int_{v_*\in E_{Pu'}} |u-v_*|^{-d+\gamma+2}b(\cos\theta)f(v_*)\; dv_*dudu', \label{delta_sigma_nc_I}
\end{align} where  $E_{pu'}=\{x \in \mathbb{R}^d|\; (x-P)\;\bot \;u-u' , \quad P=\frac{1}{\beta}u'-(\frac{1}{\beta}-1)u \;\}.$ 
In elastic model, in \cite{IM2020}, in Section 2.1, they define 
$\tilde{b}(\cos\theta)$ as
\begin{align*}
2^{d-1}b(\cos\theta)&=|v-v'|^{-(d-1)-2s}|v-v_*|^{d-2-\gamma}|v-v_*'|^{\gamma+2s+1}\tilde{b}(\cos\theta) \notag \\
&=(\frac{|v-v'|}{|v-v_*|})^{-(d-1)-2s}(\frac{|v-v_*'|}{|v-v_*|})^{\gamma+2s+1}\tilde{b}(\cos\theta) \notag \\
&=(\sin\frac{\theta}{2})^{-(d-1)-2s}(\cos\frac{\theta}{2})^{\gamma+2s+1}\tilde{b}(\cos\theta)
\end{align*} under the assumption that $b(\cos\theta) \approx_{\theta \sim 0} |\theta|^{-(d-1)-2s}.$ Here, we use $\sin\frac{\theta}{2}=\frac{|v-v'|}{|v-v_*|}$, 
$\cos\frac{\theta}{2}=\frac{|v-v_*'|}{|v-v_*|}$ when
$\cos\theta = \langle\frac{v-v_*}{|v-v_*|},\sigma\rangle$. Similarly, we define $\tilde{b}(\cos\theta)$ as
\begin{align} \label{b_nc_I}
    2^{d-2}b(\cos\theta)&= (\sin\frac{\theta}{2})^{-(d-1)-2s}(\cos\frac{\theta}{2})^{\gamma+2s+1}\tilde{b}(\cos\theta) \notag \\ &= (\frac{|u'-u|}{\beta|u-v_*|})^{-(d-1)-2s}(\frac{|u'-\beta v_*-(1-\beta)u|}{\beta|u-v_*|})^{\gamma+2s+1} \tilde{b}(\cos\theta)
\end{align} by using $\sin \frac{\theta}{2}$,  $\cos \frac{\theta}{2}$ in \eqref{cos_I}, \eqref{sin_I}. We
replace $2^{d-2}b(\cos\theta)$ in \eqref{delta_sigma_nc_I}, then 

\begin{align}
&\quad \eqref{delta_sigma_nc_I} \notag \\
    &=\beta^{2s}\int  \frac{1}{|u'-u|^{d+2s}}\delta(u'-v)(g(u)-g(u'))\int_{v_*\in E_{Pu'}} \tilde{b} (\cos\theta) |v_*+(\frac{1}{\beta}-1)u-\frac{1}{\beta}u'|^{\gamma+2s+1}f(v_*)\; dv_*dudu' \notag \\
    &=\beta^{2s}\int \frac{1}{|v-u|^{d+2s}}(g(u)-g(v))\int_{v_*\in E_{Pv}} \tilde{b} (\cos\theta) |v_*+(\frac{1}{\beta}-1)u-\frac{1}{\beta}v|^{\gamma+2s+1}f(v_*)\; dv_*du,  \notag
\end{align}
where $E_{pv}=\{x \in \mathbb{R}^d|\; (x-P)\;\bot \;u-v , \quad P=\frac{1}{\beta}v-(\frac{1}{\beta}-1)u\;\}$. 
We can deduce easily, 
\begin{align} \label{Qs_Kf_nc_I}
Q^{s}(f,g)(v) 
\doteq p.v\int K_f(u,u')\delta(u'-v)(g(u)-g(u')) \;dudu'
= p.v\int K_f(u,v)(g(u)-g(v)) \;du,
\end{align}
where 
\begin{align} \label{Kf_nc_I}
\begin{split}
    K_f(u,u') &= \frac{\beta^{2s}}{|u'-u|^{d+2s}}\int_{v_*\in E_{Pu'}} \tilde{b} (\cos\theta) |v_*+(\frac{1}{\beta}-1)u-\frac{1}{\beta}u'|^{\gamma+2s+1}f(v_*)\; dv_*\\
&=\frac{\beta^{2s}}{|u'-u|^{d+2s}}\int_{v_*\in E_{Pu'}} \tilde{b} (\cos\theta) |v_*-P|^{\gamma+2s+1}f(v_*)\; dv_*. 
\end{split}
\end{align}
The notation p.v in \eqref{Qs_Kf_nc_I} is the Cauchy principal value around the point $u'$ when $s\in[\frac{1}{2},1).$ 
The $K_f(u,u')$ is not symmetric of $K_f(u'+w,u')\neq K_f(u'-w,u')$, thus we define $\overline{K}_{f}(u,u')$ as
\begin{align} \label{K'_f_nc_I}
    \overline{K}_{f}(u,u') = \frac{\beta^{2s}}{|u'-u|^{d+2s}}\int_{v_*\in E_{Pu'}} \tilde{b} (\cos\theta) {|v_*-u'|}^{\gamma+2s+1}f(v_*)\; dv_*,
\end{align} and the above $\overline{K}_{f}(u,u')$ has the property
$\overline{K}_f(u'+w,u')= \overline{K}_f(u'-w,u')$.
The following inequality holds
\begin{align} \label{K_I_inequal}
    K_f(u,u') \leq \overline{K}_{f}(u,u'),
\end{align}
since $P$ is the closest point to $v_*$ in extension line of $u$ and $u'$.(See Figure \ref{pre_I_F}.) \vspace{3mm}

\subsubsection{Estimate on the collision operator for inelastic model}
We estimate kernel $K_f$ in \eqref{Kf_nc_I} and the collision operator $Q^{s}(f,\phi)(v)$ in inelastic model for test function $\phi(v)\in C^{2}.$ To apply symmetric property, we use $\overline{K}_f(u,u')$ instead of $K_f(u,u')$. In Lemma \ref{Kf_est_I}, we assume that  moderately soft potentials conditions, $\gamma<0$ and $\gamma+2s \in [0,2]$. In elastic model, the estimates on $Q^{s}(f,\phi)(v)$ is in Lemma 2.3, \cite{IM2020}. \\
 
\begin{lemma} ($K_f$ estimate for inelastic model) \label{Kf_est_I}
Assume that $\gamma<0$ and $\gamma+2s \in [0,2]$ (moderately soft potentials) and $M_0, E_0 < +\infty$ in \eqref{hydro_I}.
The following estimates hold. \\
(i) For any $0<r \leq 1$,
\begin{align} 
    \int_{B_{r}(u')}|u-u'|^2K_f(u,u')\; du &\leq \int_{B_{r}(u')}|u-u'|^2\overline{K}_{f}(u,u')\; du \lesssim (1+|u'|)^{\gamma+2s}r^{2-2s}, \label{Kf_est_1_I} \\
     \int_{\mathbb{R}^d/B_{r}(u')}K_f(u,u')\; du &\lesssim (1+|u'|)^{\gamma+2s}r^{-2s}. \label{Kf_est_2_I}
\end{align}
(i) For any $1<r<+\infty$,
\begin{align}  
    \int_{B_{r}(u')}|u-u'|^2K_f(u,u')\; du &\leq \int_{B_{r}(u')}|u-u'|^2\overline{K}_{f}(u,u')\; du \lesssim (1+|u'|)^{\gamma+2s}r^{\gamma+3}, \label{Kf_est_3_I} \\
     \int_{\mathbb{R}^d/B_{r}(u')}K_f(u,u')\; du &\lesssim (1+|u'|)^{\gamma+2s}r^{\gamma}. \label{Kf_est_4_I}
\end{align}
Here, $B_r(u') =\{x\in \mathbb{R}^d | \;|u'-x|<r \}$ and $K_f(u,u')$, $\overline{K}_f(u,u')$ are given in \eqref{Kf_nc_I}, \eqref{K'_f_nc_I}, respectively.
\end{lemma}
\begin{proof}
We often use the triangle inequalities that 
\begin{align} 
    |v_*-u'| &\leq |v_*-P|+(\frac{1}{\beta}-1)|u-u'|, \label{tri_inequ_Kf_est_I} \\
    |v_*-P| &\leq |v_*-u'|+(\frac{1}{\beta}-1)|u-u'|, \label{tri_inequ_Kf_est_I'} 
\end{align}
where $P=\frac{1}{\beta}u'-(\frac{1}{\beta}-1)u$.
First, we prove the inequalities \eqref{Kf_est_1_I} and \eqref{Kf_est_3_I}.
Using the inequalities \eqref{K_I_inequal} and \eqref{tri_inequ_Kf_est_I}, we have that
\begin{align}
&\quad\int_{B_{r}(u')}|u-u'|^2K_f(u,u')\; du \leq \int_{B_{r}(u')}|u-u'|^2\overline{K}_{f}(u,u')\; du \notag \\
&\lesssim
     \int_{B_{r}(u')}|u-u'|^{2-d-2s} \int_{v_*\in E_{Pu'}}|v_*-u'|^{\gamma+2s+1}f(v_*) \; dv_*du \notag \\
     &\lesssim
     \int_{B_{r}(u')}|u-u'|^{2-d-2s} \int_{v_*\in E_{Pu'}}(|v_*-P|^{\gamma+2s+1}+|u-u'|^{\gamma+2s+1})f(v_*) \; dv_*du. \label{Kf_est_I_tri_1}
\end{align}
Let $u-u'=l\vec{a}$ \;where $\vec{a} \in \mathbb{S}^{d-1}(u')$ and $v_*-P= k\vec{b}$ \;where $\vec{b} \in \mathbb{S}^{d-1}(P).$ Since $u-u'\;\bot\;v_*-P$, we add $d-2$ to the power of $k$ in \eqref{Kf_est_I_tri_1}. Next, we extend the region of $v_*$ to $\mathbb{R}^d$ and apply the triangle inequality \eqref{tri_inequ_Kf_est_I'} to  \eqref{extend_Kf_est_I}. Then we get
\begin{align}
      & \quad\eqref{Kf_est_I_tri_1} \notag \\ 
      &\lesssim \int_{\vec{a} \in \mathbb{S}^{d-1}(u')} \int_0^r l^{1-2s} \int_{\vec{b}\in \mathbb{S}^{d-1}(P)} \delta(a \cdot b)\int_0^{\infty} {k}^{\gamma+2s+d-1}f(k\vec{b}+P) \; dk db \;\; dlda +
      \int_{\vec{a} \in \mathbb{S}^{d-1}(u')}\int_0^r l^{\gamma+2} \; dlda  \notag \\
     &\lesssim  \int_0^r l^{1-2s} \int_{v_* \in \mathbb{R}^{d}} {|v_*-P|^{\gamma+2s}}f(v_*) \; dv_* \;\; dl +
     \int_0^r l^{\gamma+2} \;dl \label{extend_Kf_est_I} \\
     &\lesssim \int_0^r l^{1-2s} \int_{v_* \in \mathbb{R}^{d}} {|v_*-u'|^{\gamma+2s}}f(v_*) \; dv_*  dl + \int_0^r l^{1-2s} \int_{v_* \in \mathbb{R}^{d}} {|u-u'|^{\gamma+2s}}f(v_*) \; dv_* dl + r^{\gamma+3}  \notag\\
     & \lesssim \int_0^r l^{1-2s} \int_{v_* \in \mathbb{R}^{d}} {|v_*-u'|^{\gamma+2s}}f(v_*) \; dv_*  dl + \int_0^r l^{\gamma+1}\; dl + r^{\gamma+3} \label{int_Kf_est_I} \\ 
     &\lesssim r^{2-2s}(1+|u'|)^{\gamma+2s}+r^{\gamma+2}+r^{\gamma+3} \leq \text{max}\{r^{2-2s},r^{\gamma+2},r^{\gamma+3}\}(1+|u'|)^{\gamma+2s} \label{r_Kf_est_I}
\end{align} since $M_0, E_0 < \infty$ and $\gamma+2s \in [0,2].$ Also, we can integrate \eqref{int_Kf_est_I}, because $0<2-2s \leq \gamma+2<\gamma+3.$  
The $\text{max}\{r^{2-2s},r^{\gamma+2},r^{\gamma+3}\}$ is $r^{2-2s}$ for $0<r\leq 1$ and $r^{\gamma+3}$ for $1<r<+\infty$ in \eqref{r_Kf_est_I}. \\

\noindent The inequalities \eqref{Kf_est_2_I} and \eqref{Kf_est_4_I} are obtained similarly. We have that
\begin{align}
    &\int_{\mathbb{R}^d/B_{r}(u')}K_f(u,u')\; du
    \lesssim
     \int_{\mathbb{R}^d/B_{r}(u')}|u-u'|^{-d-2s} \int_{v_*\in E_{Pu'}}|v_*-P|^{\gamma+2s+1}f(v_*) \; dv_*du  \notag \\
    &\lesssim \int_{\vec{a} \in \mathbb{S}^{d-1}(u')} \int_r^{\infty} l^{-1-2s} \int_{\vec{b}\in \mathbb{S}^{d-1}(P)} \delta(a \cdot b)\int_0^{\infty} {k}^{\gamma+2s+d-1}f(k\vec{b}+P) \; dk db \;\; dlda  \notag \\
    &\lesssim \int_r^{\infty} l^{-1-2s} \int_{v_* \in \mathbb{R}^{d}} {|v_*-P|^{\gamma+2s}}f(v_*) \; dv_*  dl  \notag \\
    & \lesssim \int_r^{\infty} l^{-1-2s} \int_{v_* \in \mathbb{R}^{d}} {|v_*-u'|^{\gamma+2s}}f(v_*) \; dv_*  dl + \int_r^{\infty} l^{-1-2s} \int_{v_* \in \mathbb{R}^{d}} {|u-u'|^{\gamma+2s}}f(v_*) \; dv_*  dl  \notag  \\
    &\lesssim \int_r^{\infty} l^{-1-2s} \int_{v_* \in \mathbb{R}^{d}} {|v_*-u'|^{\gamma+2s}}f(v_*) \; dv_*  dl + \int_r^{\infty} l^{\gamma-1}\; dl  \label{int_Kf_est_1_I} \\
    &\lesssim r^{-2s}(1+|u'|)^{\gamma+2s}+r^{\gamma}
    \leq \text{max}\{r^{-2s},r^{\gamma}\}(1+|u'|)^{\gamma+2s} \label{r_Kf_est_1_I}
\end{align} since $M_0, E_0 < \infty$ and $\gamma+2s \in [0,2].$ To integrate \eqref{int_Kf_est_1_I}, we use the condition, $-2s \leq \gamma<0$. The $\text{max}\{r^{-2s},r^{\gamma}\}$ is $r^{-2s}$ for $0<r\leq 1$ and $r^{\gamma}$ for $1<r<+\infty$ in \eqref{r_Kf_est_1_I}.
\end{proof} \vspace{3mm}

\begin{lemma} ($Q^s$ estimate for inelastic model) \label{Q_est_I} Assume that $\gamma<0$ and $\gamma+2s \in [0,2]$ (moderately soft potentials) and $M_0, E_0 < +\infty$ in \eqref{hydro_I}. Let $\psi$ be a bounded, $C^2$ function. The $Q^{s}(f,\psi)(v)$ in \eqref{Qs_Kf_nc_I} is written by
\begin{align*}
    Q^{s}(f,\psi)(v) = \int K_f(u,u')\delta(u'-v)(\psi(u)-\psi(u')) \; dudu' =\int K_f(u,v)(\psi(u)-\psi(v)) \; du.
\end{align*} \\ 
If $\psi$ satisfies $\|\psi\|_{L^{\infty}} \leq \text{max} \{\|\nabla^2 \psi\|_{L^\infty}, \|\nabla\psi\|_{L^\infty} \}$, then 
\begin{align} \label{Q_est_1_I}
     |Q^{s}(f,\psi)(v)| 
    \lesssim
  \|\psi\|_{L^{\infty}}^{1-s}(\text{max}\{\|\nabla^2 \psi\|_{L^\infty}, \|\nabla \psi\|_{L^\infty} \})^{s}(1+|v|)^{\gamma+2s}.
\end{align} \\
Or else, if $\psi$ satisfies $\|\psi\|_{L^{\infty}} > \text{max} \{\|\nabla^2 \psi\|_{L^\infty}, \|\nabla\psi\|_{L^\infty} \}$, then
\begin{align} \label{Q_est_2_I}
     |Q^{s}(f,\psi)(v)| 
    \lesssim
    \|\psi\|_{L^{\infty}}^{1+\gamma/3}(\text{max}\{\|\nabla^2 \psi\|_{L^\infty}, \|\nabla \psi\|_{L^\infty} \})^{-\gamma/3}(1+|v|)^{\gamma+2s}.
\end{align}
\end{lemma}
\begin{proof} 
From \eqref{r_Kf_est_1_I}, we get
\begin{flalign} \label{Kf_out_Q_est_I}
    |\int \delta(u'-v)\int_{\mathbb{R}^d/B_r(u')} K_f(u,u')(\psi(u)-\psi(u')) \; dudu'| \lesssim \text{max}\{r^{-2s}, r^{\gamma}\} \|\psi\|_{L^{\infty}}(1+|v|)^{\gamma+2s}.
\end{flalign}
Now, we consider the domain $B_r(u').$ 
Using the triangle inequality and the fact that \\ $\;p.v\;\int_{B_r(v)}(u-u')\cdot \nabla\phi(u')\overline{K}_f(u,u')\;du $ is zero. Then, from \eqref{r_Kf_est_I}, we get
\begin{align}
&\quad |\int \delta(u'-v)\int_{B_r(u')} K_f(u,u')(\psi-\psi')\;dudu'|\notag \\
&\leq 
     |\int \delta(u'-v)\int_{B_r(u')} \overline{K}_{f}(u,u')(\psi-\psi')dudu'| + |\int \delta(u'-v)\int_{B_r(u')} (\overline{K}_{f}(u,u')-K_f(u,u'))(\psi-\psi')dudu'|  \notag \\ 
     &\lesssim \text{max}\{r^{2-2s}, r^{\gamma+2},  r^{\gamma+3}\} 
 \|\nabla^2 \psi\|_{L^\infty}(1+|v|)^{\gamma+2s} \notag \\
 &\quad + \|\nabla \psi\|_{L^\infty}| \underbrace{\int \delta(u'-v)\int_{B_r(u')} |\overline{K_f}(u,u')-K_f(u,u')||u-u'|\; dudu'| }_{(*)}, \label{Kf_inter_in_r_Q_est_I}
\end{align} where $\psi = \psi(u), \;\psi'=\psi(u')$.
\vspace{4mm}

\noindent Meanwhile, from \eqref{tri_inequ_Kf_est_I}, we get
\Be \label{tri_inequ_power}
    |v_*-u'|^{\gamma+2s+1} \lesssim |v_*-P|^{\gamma+2s+1}+|u-u'|^{\gamma+2s+1}.
\Ee
Then $(*)$ in the second term of \eqref{Kf_inter_in_r_Q_est_I} is estimated by  
\begin{align}
    |(*)| 
    &\lesssim |\int \delta(u'-v) \int_{B_r(u')} |u-u'|^{1-d-2s} \int_{v_* \in  E_{Pu'}} ({|v_*-u'|}^{\gamma+2s+1}-|v_*-P|^{\gamma+2s+1})f(v_*) \; dv_*dudu'|  \notag \\
    &\lesssim |\int \delta(u'-v) \int_{B_r(u')} |u-u'|^{1-d-2s} \int_{v_* \in  E_{Pu'}} {|u-u'|}^{\gamma+2s+1}f(v_*) \; dv_*dudu'|  \notag \\
	&\lesssim |\int_{B_r(v)} |u-v|^{\gamma-d+2} du| \lesssim\; \int_0^{r} l^{\gamma+1} \; dl \; \lesssim \; r^{\gamma+2}. \label{second_term_in_Q_est_I}
\end{align} \vspace{3mm}

\noindent Combining \eqref{Kf_inter_in_r_Q_est_I}
and \eqref{second_term_in_Q_est_I}, we get 
\begin{equation} \label{Kf_int_Q_est_I}
\begin{split}
&\quad|\int \delta(u'-v)\int_{B_r(u')} K_f(u,u')(\psi(u)-\psi(u')\;dudu'| \\
& \lesssim \text{max}\{r^{2-2s}, r^{\gamma+2},  r^{\gamma+3}\}
 \|\nabla^2 \psi\|_{L^\infty}(1+|v|)^{\gamma+2s}+r^{\gamma+2}\|\nabla \psi\|_{L^\infty}  \\
&\lesssim  \text{max}\{r^{2-2s}, r^{\gamma+2},  r^{\gamma+3}\} \cdot \text{max}\{\|\nabla^2 \psi\|_{L^\infty}, \|\nabla \psi\|_{L^\infty} \}(1+|v|)^{\gamma+2s}.  \\
\end{split} 
\end{equation} \vspace{3mm}

\noindent From \eqref{Kf_out_Q_est_I} and \eqref{Kf_int_Q_est_I}, we get
\begin{align} 
    |Q^{s}(f,\psi)(v)| &\leq    |\int_{\mathbb{R}^d/B_r(v)} K_f(u,v)(\psi(u)-\psi(v)) \; du| +
    |\int_{B_r(v)} K_f(u,v)(\psi(u)-\psi(v))\;du| \notag  \\ 
    &\lesssim
    \text{max}\{r^{-2s}, r^{\gamma}\} \|\psi\|_{L^{\infty}}(1+|v|)^{\gamma+2s} \label{result_Q_est_I} \\
    &\quad +\text{max}\{r^{2-2s}, r^{\gamma+2},  r^{\gamma+3}\} \cdot \text{max}\{\|\nabla^2 \psi\|_{L^\infty}, \|\nabla \psi\|_{L^\infty} \}(1+|v|)^{\gamma+2s}. \notag
\end{align} \vspace{3mm}

\noindent Choose $r=(\|\psi\|_{L^{\infty}}/\text{max}\{\|\nabla^2 \psi\|_{L^\infty}, \|\nabla \psi\|_{L^\infty} \})^{1/2}.$\\
\noindent If $0<r\leq 1$, then we obtain 
\begin{align*}
     |Q^{s}(f,\psi)(v)| &\lesssim r^{-2s}\|\psi\|_{L^{\infty}}(1+|v|)^{\gamma+2s} + r^{2-2s}\text{max}\{\|\nabla^2 \psi\|_{L^\infty}, \|\nabla \psi\|_{L^\infty} \}(1+|v|)^{\gamma+2s}\\
     &\lesssim  \|\psi\|_{L^{\infty}}^{1-s}(\text{max}\{\|\nabla^2 \psi\|_{L^\infty}, \|\nabla \psi\|_{L^\infty} \})^{s}(1+|v|)^{\gamma+2s}
\end{align*} by \eqref{result_Q_est_I}.\vspace{3mm} 

\noindent Choose $r=(\|\psi\|_{L^{\infty}}/\text{max}\{\|\nabla^2 \psi\|_{L^\infty}, \|\nabla \psi\|_{L^\infty} \})^{1/3}.$\\
\noindent If $1<r<+\infty$, then we obtain
\begin{align*}
      |Q^{s}(f,\psi)(v)| &\lesssim r^{\gamma}\|\psi\|_{L^{\infty}}(1+|v|)^{\gamma+2s} + r^{\gamma+3}\text{max}\{\|\nabla^2 \psi\|_{L^\infty}, \|\nabla \psi\|_{L^\infty} \}(1+|v|)^{\gamma+2s}\\
     &\lesssim  \|\psi\|_{L^{\infty}}^{1+\gamma/3}(\text{max}\{\|\nabla^2 \psi\|_{L^\infty}, \|\nabla \psi\|_{L^\infty} \})^{-\gamma/3}(1+|v|)^{\gamma+2s}
\end{align*} by \eqref{result_Q_est_I}. In conclusion, we estimate on $Q^{s}(f,\psi)(v)$ when  $\psi(v)\in C^2.$ 
\end{proof} \vspace{2mm}

\subsubsection{The nonsingular part, $Q^{ns}(f,f)(v)$ for inelastic model}
We prove the cancellation lemma (see \cite{AD2000} for elastic model) in inelastic model under the assumption that $\gamma > -d$. Also, we prove that $Q^{ns}(f,f)$ is well-defined and positive.

\begin{lemma} (Cancellation lemma for inelastic model) \label{CC_I}
Suppose $B(|v-v_*|,\cos a)=|v-v_*|^{\gamma}b(\cos a), \;a \in [0,\pi]$ and $\gamma > -d$ in \eqref{B_nc_I'}.
Then $Q^{ns}(f,f)(v)$ in \eqref{Qns_nc_I} is written by  
\begin{align}  \label{Lemma_CC_I_equ}
\begin{split}
     Q^{ns}(f, f)(v) &= \int f(v_*)\int B(|u-v_*|, \cos a)(\delta(u'-v)f(u')-\delta(u-v)f(u)) \; d\sigma du dv_* \\&=\int f(v_*) \int \delta(u-v)f(u)S(|u-v_*|) \; du
    dv_* 
    =f(v) \int f(v_*)S(|v-v_*|) \;
    dv_*,
\end{split}
\end{align}where  
\begin{align*}
    S(|v-v_*|)
    &= |\mathbb{S}^{d-2}||v-v_*|^{\gamma}\int_0^\pi b(\cos w)\sin^{d-2}w [(\frac{\beta}{2}\cos a + (1-\frac{\beta}{2})\cos A)^{-d-\gamma}-1] \; dw.
\end{align*}
Here, $w \doteq \sin^{-1}(\frac{\beta}{2-\beta}\sin a)+a \doteq A+a.(0<A<\frac{\pi}{2},\; 0<w,a<\pi.)$
Moreover, $Q^{ns}(f,f)(v)$ is well-defined and positive. \end{lemma}
\begin{proof}
Recall \eqref{v'_I'}, 
\begin{align} \label{v'_I''}
\begin{split}
     u'=\frac{u+v_*}{2}+\frac{1-\beta}{2}(u-v_*)+\frac{\beta}{2}\;|u-v_*|{\sigma},\\
     v_*'=\frac{u+v_*}{2}-\frac{1-\beta}{2}(u-v_*)-\frac{\beta}{2}\;|u-v_*|{\sigma}.
\end{split}
\end{align}

\begin{figure} [t]
\centering
\includegraphics[width=6cm]{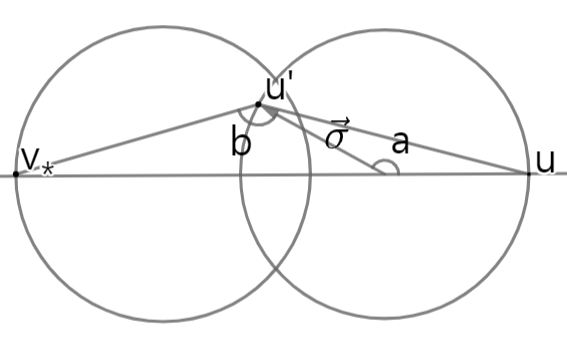} 
\caption{The relation of $u,v_*$ and $u'$ \label{CC_I_F}}
\end{figure} \noindent Let us write $k$, ${'k}$ and angles $a,b$ as
\begin{align} \label{CC_I_k_k'}
    k \cdot {\sigma} \doteq \frac{u-v_*}{|u-v_*|}\cdot {\sigma} = \cos a, \quad
    'k \cdot {\sigma} \doteq \frac{u'-v_*}{|u'-v_*|}\cdot {\sigma} = \cos b,
    \quad
    a,b \in (0,\pi). \text{\quad(See Figure \ref{CC_I_F}.)}
\end{align} Using the law of sine in Figure \ref{CC_I_F}, we get 
\begin{align} \label{CC_sine_law_I}
    \frac{(1-\frac{\beta}{2})|u-v_*|}{\sin b} = \frac{|u'-v_*|}{\sin a} =\frac{\frac{\beta}{2}|u-v_*|}{\sin(a-b)}.
\end{align}
Here, $a=b=0$, $a=b=\pi$ and if $a,b \in (0,\pi)$, then $a \neq b.$ From \eqref{CC_sine_law_I}, we get $\sin(a-b) = \frac{\beta}{2-\beta}\sin b.$ Since $0<\cos^{-1}(\frac{u'-v_*}{|u'-v_*|}\cdot \frac{u-v_*}{|u-v_*|})<\frac{\pi}{2}$, we infer $0<a-b<\frac{\pi}{2}$ and $a-b = \sin^{-1}(\frac{\beta}{2-\beta}\sin b).$ 

\noindent So, we define the map $T$ as
\begin{align} \label{CC_T_I}
    T(\cos b) = \cos(\sin^{-1}(\frac{\beta}{2-\beta}\sin b)+b) = \cos a, 
\end{align} then 
\begin{align*}
    T' = \frac{\sin a}{\sin b}(1+\frac{\frac{\beta}{2-\beta}\cos b}{\sqrt{1-(\frac{\beta}{2-\beta})^{2}\sin^{2}a}})>0 
\end{align*} since $\beta \in (\frac{1}{2},1)$ and $a,b \in (0,\pi).$ Also, we can define the function $\Psi_\sigma(u')=u$, where $v_*$ is fixed.\\ From \eqref{CC_sine_law_I}, we have that
\begin{align*}
    |\Psi_\sigma(u')-v_*|=|u-v_*|=\frac{\sin b}{\sin a}(1-\frac{\beta}{2})^{-1}|u'-v_*|=\frac{\sin b}{\sqrt{1-(T(\cos b))^2}}(1-\frac{\beta}{2})^{-1}|u'-v_*|. 
\end{align*} After renaming $u'$ to $u$ above, we obtain
\begin{align} \label{psi_replace_I}
    |\Psi_\sigma(u)-v_*|=\frac{\sin a}{\sqrt{1-(T(\cos a))^2}}(1-\frac{\beta}{2})^{-1}|u-v_*|.
\end{align} 
For the first term of $Q^{ns}(f,f)$, we perform change of variables $u' \rightarrow u$ with jacobian determinant 
\begin{align*}
|\frac{du'}{du}|=(1-\frac{\beta}{2})^{d}(1+\frac{\beta}{2-\beta}\frac{u-v_*}{|u-v_*|}\cdot \sigma)>0 
\end{align*}
and rename $u'$ to $u$ in \eqref{replace_CC_I}. We replace $|\Psi_\sigma(u)-v_*|$ in \eqref{rename_CC_I} by \eqref{psi_replace_I}, then we have that 
\begin{align} 
    & \quad \int B(|u-v_*|,\cos a) \delta(u'-v)f(u') \;d\sigma du \label{B_CC_I} \\
     &= (1-\frac{\beta}{2})^{-d}\int B(|u-v_*|,\cos a)\delta(u'-v)f(u')(1+\frac{\beta}{2-\beta}k\cdot \sigma)^{-1} \;d\sigma du' \notag \\ 
    &= (1-\frac{\beta}{2})^{-d}\int B(|\Psi_\sigma(u')-v_*|,T(\cos b)) \delta(u'-v)f(u')(1+\frac{\beta}{2-\beta}T(\cos b))^{-1} \;d\sigma du'  \label{replace_CC_I}\\
    &=(1-\frac{\beta}{2})^{-d}\int B(|\Psi_\sigma(u)-v_*|,T(\cos a)) \delta(u-v)f(u)(1+\frac{\beta}{2-\beta}T(\cos a))^{-1} \;d\sigma du \label{rename_CC_I} \\
    &=(1-\frac{\beta}{2})^{-d}\int B(\frac{\sin a}{\sqrt{1-(T(\cos a))^2}}(1-\frac{\beta}{2})^{-1}|u-v_*|,\;T(\cos a)) \delta(u-v)f(u)(1+\frac{\beta}{2-\beta}T(\cos a))^{-1} \;d\sigma du. \label{put_B_CC_I}
\end{align}
\noindent Let $w \doteq\sin^{-1}(\frac{\beta}{2-\beta}\sin a)+a$, and $0 < w < \pi.$ Note that $T(\cos a) = \cos w.$
We perform change of variables $\sigma \rightarrow a$ and $a \rightarrow w$ with jacobian determinant $|\frac{dw}{da}| = 1+\frac{\frac{\beta}{2-\beta}\cos a}{\sqrt{1-(\frac{\beta}{2-\beta})^{2}\sin^{2}a}}>0$. Using 
$B(|u-v_*|, \cos a)=|u-v_*|^\gamma b(\cos a)$, we get 
\begin{align}
    &\quad \eqref{put_B_CC_I} \notag \\ 
    &=|\mathbb{S}^{d-2}|(1-\frac{\beta}{2})^{-d}\int \int_0^\pi \sin^{d-2}a\;B(\;\frac{\sin a}{\sin w}(1-\frac{\beta}{2})^{-1}|u-v_*|,\cos w)
    \delta(u-v)f(u)(1+\frac{\beta \cos w}{2-\beta})^{-1}|\frac{da}{dw}|\; dwdu \notag \\
    &=|\mathbb{S}^{d-2}|(1-\frac{\beta}{2})^{-d-\gamma}\int \int_0^\pi |u-v_*|^{\gamma} \frac{\sin^{\gamma+d-2}a}{\sin^{\gamma}w}(1+\frac{\beta}{2-\beta}\cos w)^{-1}b(\cos w) \delta(u-v)f(u)|\frac{da}{dw}| \; dwdu. \label{before_sim_CC_I} 
\end{align} \vspace{2mm}

\noindent Let $A \doteq \sin^{-1}(\frac{\beta}{2-\beta}\sin a)$, and $0 < A < \frac{\pi}{2}.$ Note that $w = A+a.$ 
Then we can simplify in \eqref{before_sim_CC_I} by using 
\begin{align*}
    &\frac{\sin a}{\sin w} = (\frac{\beta}{2-\beta}\cos a + \cos A)^{-1}, \quad  1+\frac{\beta}{2-\beta}\cos w =\cos A\;(\frac{\beta}{2-\beta}\cos a + \cos A),\\ &|\frac{dw}{da}|=1+\frac{\frac{\beta}{2-\beta}\cos a}{\sqrt{1-(\frac{\beta}{2-\beta})^{2}\sin^{2}a}} =\frac{1}{\cos A}  (\frac{\beta}{2-\beta}\cos a + \cos A),
\end{align*} and get \eqref{after_sim_CC_I}. We write
\begin{align} \label{after_sim_CC_I}
     &\eqref{before_sim_CC_I} 
     &=|\mathbb{S}^{d-2}|\int \int_0^\pi |u-v_*|^{\gamma}b(\cos w)\sin^{d-2}w [(\frac{\beta}{2}\cos a + (1-\frac{\beta}{2})\cos A)^{-d-\gamma}\delta(u-v)f(u) \; dwdu. 
\end{align}
Also, we have that
\begin{align} \label{secondterm_CC_I}
	&\int B(|u-v_*|, \cos a)\delta(u-v)f(u) \; d\sigma du  \notag\\
	&= |\mathbb{S}^{d-2}|\int \delta(u-v)f(u)|u-v_*|^{\gamma}\;du \int_0^\pi \sin^{d-2}a \;b(\cos a)\;da   \notag \\
	&= |\mathbb{S}^{d-2}|\int \delta(u-v)f(u)|u-v_*|^{\gamma}\;du \int_0^\pi \sin^{d-2}w \;b(\cos w)\;dw.
\end{align}
Combining \eqref{after_sim_CC_I} and \eqref{secondterm_CC_I}, we deduce
\begin{equation*}
    \int B(|u-v_*|,\cos a)(\delta(u'-v)f(u')-\delta(u-v)f(u)) \;d\sigma du
     \doteq \int S(|u-v_*|)\delta(u-v)f(u) \;du, 
\end{equation*} where 
\begin{align} \label{S_nc_I}
    S(|u-v_*|)
    &= |\mathbb{S}^{d-2}||u-v_*|^{\gamma}\int_0^\pi b(\cos w)\sin^{d-2}w [(\frac{\beta}{2}\cos a + (1-\frac{\beta}{2})\cos A)^{-d-\gamma}-1] \; dw > 0.
\end{align}
Here, $w = \sin^{-1}(\frac{\beta}{2-\beta}\sin a)+a = A+a$. 
Let us check whether $S(|u-v_*|)$ is positive and well-defined. First, the following inequality holds :
\begin{align} \label{Aa}
     0 < \sin A = \frac{\beta}{2-\beta}\sin a < \sin a
     \text{\quad and \quad }
    \; 0 < |\cos a| < \cos A
\end{align}
since\;$A \in (0, \frac{\pi}{2}),\; \beta \in (\frac{1}{2}, 1).$ From \eqref{Aa} and $\beta \in (\frac{1}{2},1),$ we get 
\begin{align*}
    0 < \frac{\beta}{2}\cos a + (1-\frac{\beta}{2})\cos A < \cos A.
\end{align*} 
Due to $\gamma>-d$, we can get easily 
\begin{align*}
1< \cos^{-d-\gamma}A < (\frac{\beta}{2}\cos a + (1-\frac{\beta}{2})\cos A)^{-d-\gamma}, 
\end{align*} and $S(|u-v_*|)$ is positive. 
Using the inequality 
\begin{align*}
    1-(\frac{\beta}{2}\cos a + (1-\frac{\beta}{2})\cos A)^{d+\gamma} &\leq \text{max}\{1, d+\gamma\}(1-(\frac{\beta}{2}\cos a + (1-\frac{\beta}{2})\cos A))) \\
     &\lesssim \sin^{2}\frac{a}{2}+\sin^{2}\frac{A}{2}
 \leq 2\sin^{2}\frac{w}{2},
\end{align*} we get the $S(|u-v_*|)$ is well-defined. Here, $0<\frac{A}{2}=\frac{w-a}{2}<\frac{w}{2}<\frac{\pi}{2}$ and $0<\frac{a}{2}=\frac{w-A}{2}<\frac{w}{2}<\frac{\pi}{2}$. In conclusion, $Q^{ns}(f, f)(v)$ is written by  
\begin{align*}
    Q^{ns}(f, f)(v) &= \int f(v_*)\int B(|u-v_*|, \cos a)(\delta(u'-v)f(u')-\delta(u-v)f(u)) \; d\sigma du dv_* \\&=\int f(v_*) \int \delta(u-v)f(u)S(|u-v_*|) \;dudv_*=f(v)\int f(v_*)S(|v-v_*|)  \;dv_* > 0.
\end{align*}
Moreover, $Q^{ns}(f,f)(v)$ is well-defined and positive.  
\end{proof} 

\begin{remark}
Note that if $\alpha = 1$ in \eqref{alpha} (elastic collision), then $A=a=\frac{w}{2}$ and 
\begin{equation} \label{S_elastic}
      S(|v-v_*|)=|\mathbb{S}^{d-2}||v-v_*|^{\gamma}\int_0^\pi b(\cos w)\sin^{d-2}w(\cos^{-d-\gamma}\frac{w}{2}-1) \;dw.
\end{equation} The above equation, \eqref{S_elastic} is same $S(|v-v_*|)$ in Lemma 1 in \cite{AD2000}.  \\
\end{remark}

\subsubsection{Lowerbound for inelastic model}
If $f(t,v) \in C^{\infty}$, then $f(t,v)$ is strictly positive on $(0,+\infty)\times \mathbb{R}^{d}$ in elastic model with noncutoff kernel. (See page 103 in \cite{V2002}.)
We extend this result to the inelastic model with noncutoff kernel. 
\begin{lemma} (Positivity for inelastic model) \label{pos_nc_I}
Suppose that $f(t,v)$ in \eqref{BT_I} is $C^{\infty}$ function. Then $f(t,v)>0$ on $(0,+\infty)\times \mathbb{R}^{d}$. 
\end{lemma}
\begin{proof}
Recall \eqref{Qs_nc_I} and \eqref{Qns_nc_I}, 
\begin{align*}
\partial_t f(t,v) &=
Q(f,f)(v) = Q^{s}(f,f)(v) + Q^{ns}(f,f)(v)  \\
     &= \int B(|u-v_*|, \cos a)f(v_*)\delta(u'-v)(f(u)-f(u'))\; d\sigma dv_*du \\
     &\quad + \int B(|u-v_*|, \cos a)f(v_*)(\delta(u'-v)f(u')-\delta(u-v)f(u))\; d\sigma dv_* du.
\end{align*}
The second term in $Q^{s}(f,f)(v)$ is given by \eqref{after_sim_CC_I}, we rewrite
\begin{align}
	&\quad \int B(|u-v_*|, \cos a)\delta(u'-v)f(u')\; d\sigma du \notag \\
	&=|\mathbb{S}^{d-2}|\int \int_0^\pi |u-v_*|^{\gamma}b(\cos w)\sin^{d-2}w (\frac{\beta}{2}\cos a + (1-\frac{\beta}{2})\cos A)^{-d-\gamma}\delta(u-v)f(u) \; dwdu \notag \\
	&=f(v)\;|\mathbb{S}^{d-2}| \int_0^\pi |v-v_*|^{\gamma}b(\cos w)\sin^{d-2}w (\frac{\beta}{2}\cos a + (1-\frac{\beta}{2})\cos A)^{-d-\gamma}\; dw.  \label{second_Qs_pos_I}
\end{align}
Also, if we put $f(u')$ instead of $f(u)$ in \eqref{B_CC_I}, we change the first term in $Q^s(f,f)(v)$ as follows : 
\begin{align}
    &\quad \int B(|u-v_*|, \cos a)\delta(u'-v)f(u)\; d\sigma du \notag \\
     &=|\mathbb{S}^{d-2}|\int \int_0^\pi |u-v_*|^{\gamma}b(\cos w)\sin^{d-2}w (\frac{\beta}{2}\cos a + (1-\frac{\beta}{2})\cos A)^{-d-\gamma}\delta(u-v)f('u) \; dwdu \notag\\
   &=|\mathbb{S}^{d-2}|\int_0^\pi |v-v_*|^{\gamma}b(\cos w)\sin^{d-2}w (\frac{\beta}{2}\cos a + (1-\frac{\beta}{2})\cos A)^{-d-\gamma}f('v) \; dw, \label{first_Qs_pos_I}
\end{align}
where ${'v}$ is pre-velocity and defined 
\begin{align*} 
    v=\frac{'v+{'v_*}}{2}+\frac{1-\beta}{2}('v-{'v_*})+\frac{\beta}{2}\;|'v-{'v_*}|{\sigma}, \notag\\
     v_*=\frac{'v+{'v_*}}{2}-\frac{1-\beta}{2}('v-{'v_*})-\frac{\beta}{2}\;|'v-{'v_*}|{\sigma}.
\end{align*} Here, we consider particles with velocities $'v$ $'v_*$ before collision, $v,v_*$ after collision. Recall  \eqref{S_nc_I} in Lemma \ref{CC_I}, 
\begin{align} \label{Qns_pos_I}
    Q^{ns}(f, f)(v) =f(v)\int f(v_*)S(|v-v_*|)  \;dv_*. 
\end{align} Combining \eqref{second_Qs_pos_I}, \eqref{first_Qs_pos_I} and \eqref{Qns_pos_I}, we get 
\begin{align} 
    \partial_t f(t,v) &= Q^{s}(f,f)(v) + Q^{ns}(f,f)(v) \label{pos_nc_1_I} \\ 
&=\int f(v_*) \int \overline{S}(|v-v_*|,w)(f('v)-f(v))\; dwdv_* 
+ f(v)\int f(v_*)S(|v-v_*|)  \;dv_*, \label{pos_nc_2_I}
\end{align}
where $\overline{S}(|v-v_*|,w)=|\mathbb{S}^{d-2}||v-v_*|^{\gamma}b(\cos w)\sin^{d-2}w (\frac{\beta}{2}\cos a + (1-\frac{\beta}{2})\cos A)^{-d-\gamma}$. \\

\noindent Now, we can apply contradiction argument which is used in p.103 in \cite{V2002}. Assume that $f(t,v)$ is zero at $(t_0,v_0)$. Since $f(t_0,v) \geq 0$ for all $v$ and  $f(t_0,v_0)=0$, we get 
\begin{align*}
    \partial_t f(t,v) = 0, \quad Q^{ns}(f,f)(v) = f(v)\int f(v_*)S(|v-v_*|)  \;dv_* =0
\end{align*} at $(t,v)=(t_0,v_0).$ From \eqref{pos_nc_1_I}, we get
$Q^{s}(f,f)(v)=0$ at $(t,v)=(t_0,v_0).$ However $f(t_0,{'v})-f(t_0,v_0) \geq f(t_0,{'v}) \geq 0.$ So, we conclude $f(t_0,{'v})=f(t_0,v_0)=0$ for all ${'v} \in \mathbb{R}^d$ and it is contradiction. 
\end{proof} \vspace{3mm}

\begin{lemma} (region estimate for inelastic model) \label{region_est_I}
The velocities, $u, v_*$ and $u'$ are satisfied \eqref{u'_I}.
Suppose that $|v|\leq \sqrt{1+\beta^2}(1-\epsilon)R$ \;for\; $0 <\epsilon <1-\frac{1}{\sqrt{1+\beta^2}}$. Then, the following inequality holds 
\begin{align} \label{region_est_I_1}
    \int_{\mathbb{R}^d} \delta(u'-v)\int_{\mathbb{R}^d} \chi_{\{|u|\leq R\}} \int_{v_*\in E_{Pu'}} \chi_{\{|v_*|\leq R\}} \, dv_* \, dudu' \geq C R^{2d-1} \epsilon^d
\end{align}
for some constant, $C>0$. Here, $E_{Pu'}= \{x \in \mathbb{R}^d|\; (x-P)\bot u-u' ,\; P=\frac{1}{\beta}u'-(\frac{1}{\beta}-1)u\; \}.( \sqrt{5/
4} < \sqrt{1+\beta^2} < \sqrt{2} )$.
\end{lemma}
\begin{proof} 
Before integrating \eqref{region_est_I_1},
we find the maximum possible value of $u'$ (= $v$) where $u$ and $v_*$ are given by \eqref{u'_I} and located in $B_R(0)$. Let the point $O$ be origin. The $u$-line and $v_*$-line are perpendicular and meet the circle at exactly one point in Figure \ref{region_est_I_F_1}. Next, let the point A and H be the points of contact between the circle and $u,v_*$-lines and $\overrightarrow{AB}\bot \overrightarrow{BH}$. The $Q$-line is parallel to $v_*$-line and passes through $u'.$  \vspace{3mm}

\begin{figure} [t]
\centering
\includegraphics[width=6cm]{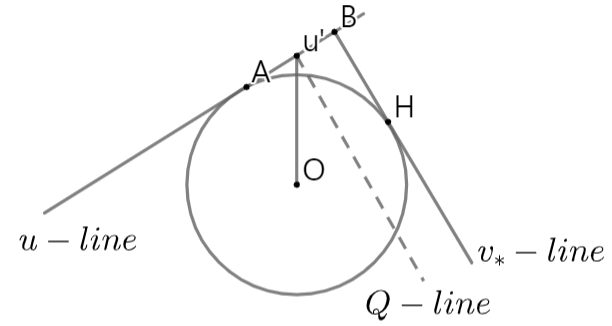} 
\caption{The maximum possible value of $u'(=v)$\label{region_est_I_F_1}} 
\end{figure} \noindent Now, let point A be $u$ and point $H$ be $v_*.$ Given $u$ and $v_*$, point Q is determined to $\beta v_* + (1-\beta)u.$ Note that $|Q-u|:|Q-v_*|=\beta : 1-\beta$ and $Q-u' \;\bot\; u-u'$ in \eqref{Q_I}. If $|\overrightarrow{OA}-u'| : |\overrightarrow{OB}-u'|=\beta : 1-\beta$, then the Q is an intersection point of the line $\overline{AH}$ and $Q$-line and the point B becomes P($=\frac{1}{\beta}v-(\frac{1}{\beta}-1)u$).
Since $|\overrightarrow{AB}|=R$, we can find $|\overrightarrow{OA}-u'|=\beta R$ and $|u'|=\sqrt{1+\beta^2}R$. So, the range of the possible value of $u'$ is between R and $\sqrt{1+\beta^2}R.$ \vspace{3mm}

\noindent In Figure \ref{region_est_I_F_2}, let $|u'| = \sqrt{1+\beta^2}(1-\epsilon)R$ and $|\overrightarrow{OW}|$ be $\sqrt{1+\beta^2}R.$
Note that $u'$ is moved down as much as $\sqrt{1+\beta^2}\epsilon$ in Figure \ref{region_est_I_F_1}. Suppose that $\overrightarrow{AW}$ and $\overrightarrow{OA'}-u'$ are parallel then \;$|\overrightarrow{AC}|=R-(1-\epsilon)R=\epsilon R$ and $|\overrightarrow{CD}|^{2}=R^2-(1-\epsilon)^2R^2 \backsimeq \epsilon R^2.$
We denote the plane $E_{AC}$ that contains $\overrightarrow{AC}$ with normal vector, $\overrightarrow{CD}.$ Similarly, we denote plane $E_{CD}$ that contains $\overrightarrow{CD}$ with normal vector $\overrightarrow{AC}$ and  $E_{EF}$ that contains $\overrightarrow{EF}$ with normal vector $\overrightarrow{AB}.$ We also denote $B_{R}(0)'=\{x \in B_R(0) | 0<\cos^{-1}(\frac{x \cdot u'}{|x||u'|})<\cos^{-1}(\frac{\overrightarrow{OA} \cdot u'}{|\overrightarrow{OA}||u'|})\}$. \vspace{3mm}

\begin{figure}[t]
\centering
\includegraphics[width=5cm]{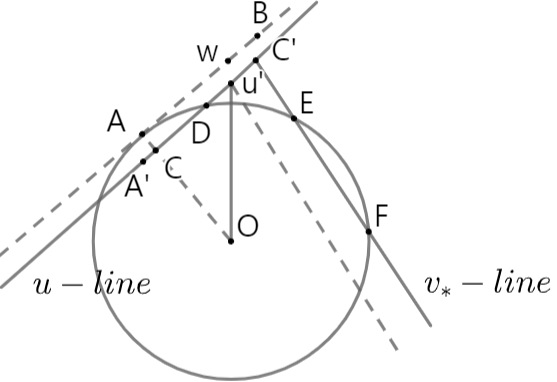}
\caption{The possible region of $u,v_*$ when $u'$ is fixed.\label{region_est_I_F_2}}
\end{figure} 

\noindent We prove that
if $u$ is a fixed point in the place surrounded by the planes $E_{AC}$ and $E_{CD}$ and $B_R(0)'$, then the possible region of $v_*$ is greater than $E_{EF} \cap B_R(0)$. First, when $u$ is fixed by the point $A'$, then $v_*$ is located in the plane $E_{EF} \cap B_R(0)$ since $|\overrightarrow{OA'}-u'|:|\overrightarrow{OC'}-u'|=\beta : 1-\beta$. If $u$ moves toward $u'$ on line $\overline{A'D}$, then $v_*$-line of the set of the possible $v_*$ moves toward $O$ in parallel to the $\overrightarrow{AB}.$ Also, we consider the clockwise  rotation of the $u,v_*$-lines when $u'$ is fixed as a center in Figure \ref{region_est_I_F_2}. 
For some $x$ on the $\overline{A'D}$, \;if the angle $\cos^{-1}(\frac{(u'-x) \cdot u'}{|u'-x||u'|})$ increases while maintaining $|x-u'|$ value, then the $v_*$-line rotates close to O. Therefore, both cases, the possible region of $v_*$ increases. \vspace{3mm}

\noindent The volume of the place surrounded by the planes $E_{AC}$ and $E_{CD}$ and $B_R(0)'$ is proportional to $|\overrightarrow{AC}| \times |\overrightarrow{CD}|^{d-1}$($\varpropto$ $R\epsilon  \times (R\sqrt{\epsilon})^{d-1}$). The volume of the $E_{EF} \cap B_R(0)$ is proportional to $\overline{EF}^{d-1} (\varpropto (R\sqrt{\epsilon})^{d-1})$. Lastly, the following inequality holds, 
\begin{align*}
    \int_{\mathbb{R}^d}\int_{\mathbb{R}^d} \chi_{\{|u|\leq R\}}\;\delta(u'-v)\int_{v_*\in E_{Pu'}} \chi_{\{|v_*|\leq R\}} \, dv_* \, dudu' \;\geq \; C (R\epsilon)(R\sqrt{\epsilon})^{d-1}(R\sqrt{\epsilon})^{d-1}\; \geq C R^{2d-1} \epsilon^{d}
\end{align*} for some constant $C$. 
\end{proof} \vspace{3mm}

The spreading lemma in elastic model has been proved in Lemma 3.4 in \cite{IM2020} and we extend this result to inelastic model.

\begin{lemma} (Spreading lemma for inelastic model) \label{spreading_I} Consider $T_0$ $\in (0,1).$ Suppose that $\gamma<0$ and $\gamma+2s \in [0,2]$ (moderately soft potentials) and $f(t,v)$ is the solution in \eqref{BT_I}. 
If $f(t,v) \geq l$ on $ t \in [0,T_0], \; v\in B_R(0)$ \;for some $l>0, R>1$ then there exists constant $C>0$ depending on $d, s, M_0, E_0$ such that
\begin{align*}
    f(t,v) \geq C\; min(t, R^{-\gamma}\epsilon^{2s})\; \epsilon^q R^{d+r}l^2, \quad \forall t \in [0,T_0], \;\;
    \forall v\in B_{\sqrt{1+\beta^2}(1-\epsilon)R}(0)
\end{align*}
for any $\epsilon \in (0, 1-\frac{1}{\sqrt{1+\beta^2}})$ that satisfied with $\epsilon^{q}R^{d+r}l < 1/2$ and $R\epsilon < 1$, where $q=d+2(\gamma+2s+1).$
\end{lemma}
\begin{proof} 
We will obtain this lemma using Lemma \ref{Q_est_I}, \ref{CC_I}, and \ref{region_est_I}. \vspace{3mm}

\noindent We consider the smooth function  $f(u)$ that $f(0)=1$ and $f=0$ in $|u|>1$ and for which the following is true, 
\begin{align} \label{f_condition}
    \frac{\sqrt{2}}{2}\|\nabla f\|_{L^\infty} \leq  \|\nabla^2 f \|_{L^\infty}, \quad \frac{1}{2} < \|\nabla^2 f \|_{L^\infty}.
\end{align} (ex. $f(x)= \exp({1-\frac{1}{1-|x|^2}})$).
From $f(u),$ we can derive function $\phi_{R,\epsilon}(u),$
\begin{align} \label{phi_spreading_nc_I}
    \phi_{R,\epsilon}(u) =
     \left\{ \begin{array}{rcl}1 & \mbox{for} \; |u|\leq\sqrt{1+\beta^2}R(1-\epsilon)\\
0 & \mbox{\quad for} \;|u|> \sqrt{1+\beta^2}R(1-\epsilon/2),
\end{array}\right.
\end{align} such that $\| \nabla\phi_{R,\epsilon}\|_{L^\infty}=(\frac{\sqrt{1+\beta^2}R\epsilon}{2})^{-1}\|\nabla f\|_{L^\infty}$ and $\|\nabla^2 \phi_{R,\epsilon}\|_{L^\infty}=(\frac{\sqrt{1+\beta^2}R\epsilon}{2})^{-2}\|\nabla^2 f \|_{L^\infty}.$ \\ By \eqref{f_condition} and $R\epsilon <1$, and $\beta \in (\frac{1}{2},1)$, we get
\begin{align} \label{phi condition}
\begin{split}
    &\| \nabla\phi_{R,\epsilon}\|_{L^\infty} \leq \|\nabla^2 \phi_{R,\epsilon}\|_{L^\infty}, \\
    &\|\phi_{R,\epsilon}\|_{L^{\infty}}/\text{max}\{\|\nabla^2 \phi_{R,\epsilon}\|_{L^\infty}, \|\nabla \phi_{R,\epsilon}\|_{L^\infty} \}< 1. 
\end{split}
\end{align} So, we can apply \eqref{Q_est_1_I} and obtain
\begin{align}
  |Q^{s}(f,\phi_{R,\epsilon})(v)| 
    &\lesssim
  \|\phi_{R,\epsilon}\|_{L^{\infty}}^{1-s}(\text{max}\{\|\nabla^2 \phi_{R,\epsilon}\|_{L^\infty}, \|\nabla \phi_{R,\epsilon}\|_{L^\infty} \})^{s}(1+|v|)^{\gamma+2s}  \notag \\
    &= \|\phi_{R,\epsilon}\|_{L^{\infty}}^{1-s}\|\nabla^2 \phi_{R,\epsilon}\|_{L^\infty}^{s}(1+|v|)^{\gamma+2s}
  \lesssim cR^\gamma \epsilon^{-2s} \label{Qs_est_spreading_I}
\end{align}for some constant $c>0$. \vspace{4mm}

\noindent We want to prove $ f(t,v) \geq \tilde{l}(t)\phi_{R,\epsilon}(v)$ such that
\begin{align} \label{l_spreading_I}
    \tilde{l}(t)=\alpha \; \epsilon^{q}R^{d+\gamma}l^2\frac{1-e^{-c R^\gamma\epsilon^{-2s}t}}{c R^{\gamma}\epsilon^{-2s}}
\end{align} for some $\alpha >0$. If the inequality was not true, there exists $(t_0,v_0)\in [0,T_0]\times \mbox{supp}\;\phi_{R,\epsilon}$ such that $f(t_0,v_0)=\tilde{l}(t_0)\phi_{R,\epsilon}(v_0)$ and $\partial_t(f-\tilde{l}(t)\phi_{R,\epsilon}(v))\leq 0$ at $(t_0,v_0).$ 
Applying $Q^{ns}(f,f)>0$ and \eqref{Qs_est_spreading_I},
\begin{align}
    \tilde{l}'(t_0) &\geq \tilde{l}'(t_0)\phi_{R,\epsilon}(v_0) \geq \partial_t f(t_0,v_0) \geq Q^{s}(f,f)(t_0,v_0)  \notag \\
     &=Q^{s}(f,f-\tilde{l}\phi_{R,\epsilon})(t_0,v_0)+\tilde{l}(t_0)Q^{s}(f,\phi_{R,\epsilon})(t_0,v_0) \notag \\
     &\geq \int \delta(u'-v_0)\int K_f(u,u')(f(u)-\tilde{l}(t_0)\phi_{R,\epsilon}(u))\; dudu' - cR^{\gamma}\epsilon^{-2s}\tilde{l}(t_0). \label{app_Q_spreading_I}
\end{align}\vspace{2mm}

\noindent Using $\tilde{l}'(t_0)=\alpha \epsilon^q R^{d+r}l^2-cR^{\gamma}\epsilon^{-2s}\tilde{l}(t_0),$ we obtain 
\begin{align}
    \alpha \epsilon^q R^{d+r}l^2  \gtrsim \int \delta(u'-v_0)
    \int \frac{f(u)-\tilde{l}(t_0)\phi_{R,\epsilon}(u)}{|u-u'|^{d+2s}}\int_{v_*\in E_{Pu'}} |v_*+(\frac{1}{\beta}-1)u-\frac{1}{\beta}u'|^{\gamma+2s+1}f(v_*)\; dv_*dudu'.  \label{est_spreading_I} 
\end{align}
Restricting $u$ and $v_*$ to $B_R(0)$, we get $f(u)-\tilde{l}(t_0)\phi(u)\geq l/2$ under the assumption that $\epsilon^{q}R^{d+r}l < 1/2$, and  $|u'-u|\leq 3R$. Then, we obtain \eqref{region_est_spreading_I}. Next we have that $|v_*-(1-\frac{1}{\beta})u-\frac{1}{\beta}u'| \gtrsim \epsilon^2R$ in \eqref{region_est_spreading_I} when $u'$ and $u$ are fixed. Because the form of the possible region of $v_*$ is not changed(i.e, $\varpropto (R\sqrt{\epsilon})^{d-1})$ in Lemma \ref{region_est_I_F_1}, we can apply \eqref{region_est_I_1} to \eqref{region_est_spreading_I}. Thus,
\begin{align}
     &\eqref{est_spreading_I} \notag \\
     &\gtrsim  R^{-d-2s}l^2 \int \delta(u'-v_0)\int \chi_{\{|u|\leq R\}}\;\int_{v_*\in E_{Pu'}} |v_*+(\frac{1}{\beta}-1)u-\frac{1}{\beta}u'|^{\gamma+2s+1} \chi_{\{|v_*|\leq R\}} \, dv_* \, dudu'  \label{region_est_spreading_I}\\
    &\geq C'\epsilon^q R^{d+\gamma}l^2 \notag
\end{align} on $v_0<\sqrt{1+\beta^2}R(1-\epsilon)$ for some constant $C'>0.$ Here, we can choose the constant $C'$ which is independent of $\alpha$ in \eqref{l_spreading_I}. This is contradiction.
\end{proof} \vspace{3mm}

\begin{proof} [Proof of Theorem \ref{1.1}]
For any $T_0\in (0,1),$ we define
\begin{align*}
     &T_n = (1-\frac{1}{2^{n}})T_0, \\
     &\epsilon_n=\frac{1}{2^{n+1}}, \\
      &R_{n+1}=\sqrt{1+\beta^2}(1-\epsilon_n)R_n, \quad R_0=1.
\end{align*}
In Lemma \ref{pos_nc_I}, we proved $f$ is strictly positive under the assumption that $f \in C^{\infty}.$ Because the interval $[T_1,T_0]$ and $\overline{B}_{R_0}(0)$ are compact, there exists some constant $l_0\in (0,1)$ such that $f(t,v) \geq l_0$ when $t \in [T_1,T_0],\; v \in B_{R_0}(0).$  
For $v \in B_{R_n}(0)$, we assume $f(t,v) \geq l_n$ when $t \in [T_{n+1},T_0]$.  We check $\epsilon_n^q R_n^{d+\gamma}l_n < 1/2$ and $R_n\epsilon_n <1.$ Then, we get $f(t,v) \geq l_{n+1}$ when $t \in [T_{n+2},T_0], \; v \in B_{R_{n+1}}(0) $ and $l_{n+1} \geq K l_n^{2}$ for some constant $K$ depending on $d, s, M_0, E_0, t$ by Lemma \ref{spreading_I}. Now, using iteration, we obtain 
\Be \label{iteration}
	l_n \geq Kl_{n-1}^{2} \cdots \geq Kl_{0}^{2^{n}}. 
\Ee
 \noindent We observe that
\begin{align*}
R_{n}=\sqrt{1+\beta^2}(1-\epsilon_n)R_{n-1}=(\sqrt{1+\beta^2})^{n}\Pi_{j=1}^{n}(1-\frac{1}{2^j}).
\end{align*} 
Then we get
\begin{equation*}
    C_2(\sqrt{1+\beta^2})^{n} \leq R_n \leq C_1(\sqrt{1+\beta^2})^{n},
\end{equation*}
\begin{equation*}
    C_2 \cdot2^{\frac{1}{k}} \leq R_n \leq C_1 \cdot2^{\frac{1}{k}}
\end{equation*}
for some constants, $C_1, C_2>0.$ Here, $k=\log 2/(n\log\sqrt{1+\beta^2}).$ If we write $p=\log 2/\log\sqrt{1+\beta^2}$, $l_0^{2^n}$ on the RHS of \eqref{iteration} can be estimated by 
%We develop the form of $l_0^{2^n},$ 
%\begin{align*} 
\Be \label{l v p}
    l_0^{2^n} \geq l_0^{{(\frac{R_n}{C_2})}^{kn}} 
    \geq e^{-C R_n^{p}} \geq e^{-C |v|^{p}},
\Ee
%\end{align*}
where $C=-({\frac{1}{C_2}})^p \log l_0.$ Therefore, combining \eqref{iteration} and \eqref{l v p}, we get $f(T_0,v) \geq K e^{-C |v|^{p}}$ on $v \in \mathbb{R}^d,$ where $K$ and $C$ depend on $d,s,M_0,E_0$ and $T_0$. 
Extending this result, there exist functions $a(t), b(t)>0$ such that $f(t,v) \geq a(t)e^{-b(t)|v|^{p}}$ for all positive $t.$ 
\end{proof} \vspace{2mm}

\section{The Mixture Boltzmann equation}
\subsection{Noncutoff collision kernel}
The Mixture collision operator $Q_{ji}(f_j,f_i)(t,v)$ is written by  
\begin{align} \label{Q_nc_M'}
\begin{split}
     Q_{ji}(f_j, f_i) (t,v) &= \int B_{ji}(|v-v_*|, \cos\theta)(f_j(v_*')f_i(v')-f_j(v_*)f_i(v)) \; d\sigma dv_*, \\ \quad &\text{where}\;\cos\theta = \langle\frac{v-v_*}{|v-v_*|}, \sigma\rangle, \quad \theta \in [0, \pi],
\end{split}
\end{align}
and $v'$ and $v_*'$ are given by \eqref{v'_M}. 
We define the $B_{ji}(|v-v_*|,\cos\theta)$ as
\begin{equation} \label{B_nc_M'}
    B_{ji}(|v-v_*|,\cos\theta)=|v-v_*|^{\gamma}b_{ji}(\cos\theta), \;\; b_{ji}(\cos\theta) \approx_{\theta \sim 0} |\theta|^{-(d-1)-2s} \tilde{b}_{ji}(\cos\theta), \;\;
    \int \tilde{b}_{ji}(\cos\theta) \;d\sigma < +\infty.
\end{equation}
The $\tilde{b}_{ji}(\cos\theta)$ is smooth on $[0,\pi]$ and positive on $[0,\pi).$ Because taking minus on ${\sigma}$ does not reverse $v'$ and $v_*'$($v' \nleftrightarrow v_*'$ when ${\sigma} \leftrightarrow -{\sigma}$), we can not reduce the range of $\theta$, $[0, \pi]$ to $[0, \frac{\pi}{2}]$. The collision is time-reversible, i.e, $B_{ji}(|v-v_*|,\cos\theta)=B_{ji}(|v'-v_*'|,\cos\theta)$. \vspace{3mm}

\noindent The weak form of the collision operator is written by 
\begin{align} \label{weak_nc_M}
    \int Q_{ji}(f_j,f_i)\phi(v) \;dv = \int B_{ji}(|v-v_*|,\cos\theta)f_j(v_*)f_i(v)(\phi(v')-\phi(v)) \;d\sigma dv_* dv.
\end{align} 
\noindent 
When $j=i$, $Q_{ii}(f_i,f_i)$ is the collision operator between identical particles under elastic collision. The $Q_{ii}(f_i,f_i)$ can be decomposed into singular and nonsingular parts,
\[
	Q_{ii}(f_i,f_i) = Q_{ii}^{s}(f_i,f_i)(t,v) + Q_{ii}^{ns}(f_i,f_i)(t,v).
\]
The estimates for $Q_{ii}^{s}(f_i,\phi)(t,v)$ is given in Lemma 2.3, \cite{IM2020} and the cancellation lemma of the $Q_{ii}^{ns}(f_i,f_i)(t,v)$ is in \cite{V2002}. Since we are studying multi-species models, we extend the results into $Q_{ji}(f_j,f_i)(t,v)$ for $i \neq j$. \\ %This section is not very different from the inelastic model, so we will write it more briefly. \vspace{3mm}

\noindent First, we assume that $m_i<m_j.$ We decompose $Q_{ji}(f_j,f_i)(t,v)$ into singular and nonsingular parts,
\begin{align} 
& \quad\;\; Q_{ji}(f_j,f_i)(t,v) \notag \\  &= \int B_{ji}(|v-v_*|,\cos\theta)(f_j(v_*')f_i(v')-f_j(v_*)f_i(v)) \; d\sigma dv_* \notag \\
 &=  \int B_{ji}(|v-v_*|,\cos\theta)f_j(v_*')(f_i(v')-f_i(v))\;d\sigma dv_* +  f_i(v)\int B_{ji}(|v-v_*|,\cos\theta)(f_j(v_*')-f_j(v_*)) \; d\sigma dv_* \notag \\
 &\doteq Q^{s}_{ji}(f_j,f_i)(t,v) + Q^{ns}_{ji}(f_j,f_i)(t,v).  \label{split_Q_M_1}
\end{align}
We define 
\begin{align}
    Q_{ji}^{s}(f_j,f_i)(t,v) &= \int B_{ji}(|v-v_*|,\cos\theta)f_j(v_*')(f_i(v')-f_i(v))\;d\sigma dv_* \label{Qs_M_1}, \\
    Q^{ns}_{ji}(f_j,f_i)(t,v) &= f_i(v)\int B_{ji}(|v-v_*|,\cos\theta)(f_j(v_*')-f_j(v_*)) \; d\sigma dv_*\label{Qns_M_1}.
\end{align}
We change $Q_{ji}^{s}(f_j,f_i)(t,v)$ into the Carleman alternative representation form. We perform change of variables $(\sigma, v_*) \rightarrow (v',v_*')$ and replace $\delta(|\sigma|^2-1)$ by \eqref{delta_M_1}. We define $\tilde{b}_{ji}(\cos\theta)$ as
\begin{align*}
    2^{d-1}b_{ji}(\cos\theta)&= (\sin\frac{\theta}{2})^{-(d-1)-2s}(\cos\frac{\theta}{2})^{\gamma+2s+1}\tilde{b}_{ji}(\cos\theta) \\ &= (\frac{|v'-v|}{\frac{2m_j}{m_i+m_j}|v-v_*|})^{-(d-1)-2s}(\frac{|v-\frac{2m_j}{m_i+m_j}v_*'+\frac{m_j-m_i}{m_i+m_j}v'|}{\frac{2m_j}{m_i+m_j}|v-v_*|})^{\gamma+2s+1} \tilde{b}_{ji}(\cos\theta)
\end{align*} under the assumption that $b_{ji}(\cos\theta)\approx|\theta|^{-(d-1)-2s}\tilde{b}_{ji}(\cos\theta)$. Here, $\sin {\frac{\theta}{2}}$, and $\cos {\frac{\theta}{2}}$ are in \eqref{cos_M_1}, \eqref{sin_M_1}. Thus,
\begin{align} 
    &\quad \int B_{ji}(|v-v_*|,\cos\theta)f_j(v_*')(f_i(v')-f_i(v))\;d\sigma dv_* \label{Qs_start_M_1} \\
     &=\int (\frac{m_j}{m_i+m_j}|v-v_*|)^{-d}|v-v_*|^{\gamma}b_{ji}(\cos\theta)f_j(v_*')(f_i(v')-f_i(v))\delta(|\sigma|^2-1) \; dv'dv_*' \notag \\ 
      &= (\frac{m_j}{m_i+m_j})^{-d+1}\int |v'-v|^{-1}(f_i(v')-f_i(v)) \int_{v_*'\in E_{Pv'}} |v-v_*|^{-d+\gamma+2}b_{ji}(\cos\theta)f_j(v_*')\; dv_*'dv' \notag \\ 
       &=(\frac{2m_j}{m_i+m_j})^{2s}\int \frac{f_i(v')-f_i(v)}{|v'-v|^{d+2s}}\int_{v_*'\in E_{Pv'}}  \tilde{b}_{ji} (\cos\theta) |\frac{m_i+m_j}{2m_j}v+\frac{m_j-m_i}{2m_j}v'-v_*'|^{\gamma+2s+1}f_j(v_*')\; dv_*'dv',\notag
\end{align}
where $ E_{Pv'} = \{x \in \mathbb{R}^d|\; (x-P)\;\bot \;v-v' , \quad P=\frac{m_i+m_j}{2m_j}v+\frac{m_j-m_i}{2m_j}v'\; \}.$ We can deduce easily, 
\begin{align} \label{Qs_Kf_j_M_1} 
Q^{s}_{ji}(f_j,f_i)(t,v) &\doteq p.v \int K_{f_j}(v,v')(f_i(v')-f_i(v)) \;dv',
\end{align}
where 
\begin{align}  \label{Kf_j_M_1} 
\begin{split}
    K_{f_j}(v,v') &= (\frac{2m_j}{m_i+m_j})^{2s}\frac{1}{|v'-v|^{d+2s}}\int_{v_*'\in E_{Pv'}} \tilde{b}_{ji}(\cos\theta) |\frac{m_i+m_j}{2m_j}v+\frac{m_j-m_i}{2m_j}v'-v_*'|^{\gamma+2s+1}f_j(v_*')\; dv_*'\\
    &= (\frac{2m_j}{m_i+m_j})^{2s}\frac{1}{|v'-v|^{d+2s}}\int_{v_*'\in E_{Pv'}} \tilde{b}_{ji}(\cos\theta) |P-v_*'|^{\gamma+2s+1}f_j(v_*')\; dv_*'. 
\end{split}
\end{align} The notation p.v in \eqref{Qs_Kf_j_M_1} is the Cauchy principal value around the point $v'$ when $s\in[\frac{1}{2},1).$
The $K_{f_j}(v,v')$ is not symmetric of $K_{f_j}(v,v+w)\neq K_{f_j}(v,v-w)$. We define $\overline{K}_{f_j}(v,v')$ which has symmetric property,
\begin{align} \label{K'f_j_M_1}
    \overline{K}_{f_j}(v,v') = (\frac{2m_j}{m_i+m_j})^{2s}&\frac{1}{|v'-v|^{d+2s}}\int_{v_*'\in E_{Pv'}} \tilde{b}_{ji}(\cos\theta) |v-v_*'|^{\gamma+2s+1}f_j(v_*')\; dv_*'.
\end{align} The following inequality holds
\begin{align} \label{K_M_1_inequ}
    K_{f_j}(v,v') \leq \overline{K}_{f_j}(v,v'),
\end{align} since $P$ is the closest point to $v_*$ in extension line of $v'$ and $v$. (See Figure \ref{Pre_M_F_1}.) \vspace{3mm}

\noindent Next, we assume that $m_i>m_j$. 
We split $Q_{ji}(f_j,f_i)(v)$, using the weak form \eqref{weak_nc_M},
\begin{align}
   Q_{ji}(f_j,f_i)(v)
   &=\int Q_{ji}(f_j,f_i)(u)\delta(u-v) \;du \notag \\
     &= \int B_{ji}(|u-v_*|, \cos\theta)f_i(u)f_j(v_*) (\delta(u'-v)-\delta(u-v)) \;d\sigma dv_* du \notag\\
    &\doteq Q^{s}_{ji}(f_j,f_i)(v) + Q^{ns}_{ji}(f_j,f_i)(v).\label{Qs_Qns_M_2}
\end{align}
We define
\begin{align} 
    Q^{s}_{ji}(f_j,f_i)(v) := \int B_{ji}(|u-v_*|, \cos\theta)f_j(v_*)(\delta(u'-v)f_i(u)-\delta(u'-v)f_i(u'))\; d\sigma dv_*du, \label{Qs_M_2} \\
    Q^{ns}_{ji}(f_j,f_i)(v) := \int B_{ji}(|u-v_*|, \cos\theta)f_j(v_*)(\delta(u'-v)f_i(u')-\delta(u-v)f_i(u))\; d\sigma dv_* du.\label{Qns_M_2}
\end{align}
We change the $Q_{ji}^{s}(f_j,f_i)(t,v)$ into the Carleman alternative representation form. Using similar argument to inelastic model (See page 9.), we get
\begin{align} \label{Qs_ji_Cal_M_2}
\begin{split}
     Q_{ji}^{s}(f_j,f_i)(v) &= p.v \int K'_{f_j}(u,u')\delta(u'-v)(f_i(u)-f_i(u')) \;dudu'  \\ 
     &= p.v \int K'_{f_j}(u,v)(f_i(u)-f_i(v))\;du,
\end{split}
\end{align}
where
\begin{align} 
    K'_{f_j}(u,u')&=(\frac{2m_j}{m_i+m_j})^{2s}\frac{1}{|u'-u|^{d+2s}} \int_{v_* \in E_{Ru'}}\tilde{b}_{ji}(\cos\theta)|\frac{m_j+m_i}{2m_j}u'-\frac{m_i-m_j}{2m_j}u-v_*|^{\gamma+2s+1} f_j(v_*) \;dv_* \notag \\
    &=(\frac{2m_j}{m_i+m_j})^{2s}\frac{1}{|u'-u|^{d+2s}} \int_{v_* \in E_{Ru'}}\tilde{b}_{ji}(\cos\theta)|R-v_*|^{\gamma+2s+1} f_j(v_*) \;dv_*. \label{Kf_M_2}
\end{align} Here, $E_{Ru'}=\{x|\; (x-R)\;\bot u-u', \quad R=\frac{m_i+m_j}{2m_j}u'-\frac{m_i-m_j}{2m_j}u\;\}.$
The notation p.v in \eqref{Qs_ji_Cal_M_2} is the Cauchy principal value around the point $u'$ when $s\in[\frac{1}{2},1).$ We define $\overline{K}'_{f_j}(u,u')$ which has symmetric property $\overline{K}'_{f_j}(u'+w,u')= \overline{K}'_{f_j}(u'-w,u')$ as
\begin{align} \label{K'f_j_M_2}
    \overline{K}'_{f_j}(u,u')=(\frac{2m_j}{m_i+m_j})^{2s}&\frac{1}{|u'-u|^{d+2s}} \int_{v_* \in E_{Ru'}}\tilde{b}_{ji}(\cos\theta)|v_*-u'|^{\gamma+2s+1} f_j(v_*) \;dv_*.
\end{align} 
The following inequality holds 
\begin{align} \label{K_M_2_inequ}
    {K}'_{f_j}(u,u')\leq \overline{K}'_{f_j}(u,u'),
\end{align} since $R$ is the closest point to $v_*$ in extension line with $u'$ and $u$ from $v_*$.(See Figure \ref{Pre_M_F_2}.) \vspace{3mm}

\subsubsection{Estimate on the collision operator for mixture model}
We estimate kernel $K_f$ in \eqref{Kf_j_M_1} (also in \eqref{Kf_M_2}) and the collision operator $Q^{s}_{ji}(f_j,\phi)(v)$ in mixture model when  $\phi(v)\in C^{2}$ is a test function. These estimates are similar as the inelastic model case of  Section 3.1.1.  We will consider two cases : $m_i<m_j$ and $m_i>m_j$ separately. \vspace{2mm}

\begin{lemma} ($K_{f_j}$ estimate for $m_i < m_j$) \label{Kf_j_estimate_M_1}
Assume that $\gamma<0$ and $\gamma+2s \in [0,2]$(moderately soft potentials) and $E_0,M_0<+\infty$ in \eqref{hydro_M}. Let $f_i$, $f_j$ be the density of particles of mass $m_i$, $m_j$ and $m_i<m_j.$ The following estimates hold. \\
(i) For any $0<r \leq 1$,
\begin{align*} 
    \int_{B_{r}(v)}|v-v'|^2K_{f_j}(v,v')\; dv' &\leq \int_{B_{r}(v)}|v-v'|^2\overline{K}_{f_j}(v,v')\; dv' \lesssim (1+|v|)^{\gamma+2s}r^{2-2s}, \\
     \int_{\mathbb{R}^d/B_{r}(v)}K_{f_j}(v,v')\; dv' &\lesssim (1+|v|)^{\gamma+2s}r^{-2s}.
\end{align*}
(ii) For any $1<r<+\infty$,
\begin{align*}  
    \int_{B_{r}(v)}|v-v'|^2K_{f_j}(v,v')\; dv' &\leq \int_{B_{r}(v)}|v-v'|^2\overline{K}_{f_j}(v,v')\; dv' \lesssim (1+|v|)^{\gamma+2s}r^{\gamma+3}, \\
     \int_{\mathbb{R}^d/B_{r}(v)}K_{f_j}(v,v')\; dv' &\lesssim (1+|v|)^{\gamma+2s}r^{\gamma}.
\end{align*}
Here, $B_r(v) =\{x\in \mathbb{R}^d | \;|v-x|<r \}$ and
$K_{f_j}(v,v')$ and $\overline{K}_{f_j}(v,v')$ are given in \eqref{Kf_j_M_1} and \eqref{K'f_j_M_1}, respectively.
\end{lemma}
\begin{proof} In Lemma \ref{Kf_est_I}, we rename $v_*$ to $v_*'$ and $u$ to $v'$ and $u'$ to $v$, and use the point P defined in \eqref{Pre_P_M_1} instead of \eqref{P_I}. 
Also, we use the following triangle inequalities, 
\begin{align} 
    &|v-v_*'| \leq |v_*'-P|+|P-v| = |v_*'-P|+\frac{m_j-m_i}{2m_j}|v-v'|, \label{Kf_inequal_M_1} \\
    &|v_*-P| \leq |v-v_*'|+\frac{m_j-m_i}{2m_j}|v-v'|, \label{Kf_inequal_M_1'}
\end{align} instead of \eqref{tri_inequ_Kf_est_I}, \eqref{tri_inequ_Kf_est_I'}.
Here, the point $P$ is also defined in \eqref{Pre_P_M_1}. Then, the proof is almost similar to Lemma \ref{Kf_est_I} and we omit the details.
\end{proof} \vspace{3mm}

\begin{lemma} ($K_{f_j}$ estimate, $m_i>m_j$) \label{Kf_j_estimate_M_2}
Assume that $\gamma<0$ and $\gamma+2s \in [0,2]$(moderately soft potentials) and $E_0,M_0<+\infty$ in \eqref{hydro_M}. Let $f_i$, $f_j$ be the density of particles of mass $m_i$, $m_j$ and $m_i>m_j.$ The following estimates hold.
\\
(i) For any $0<r \leq 1$,
\begin{align*} 
 \int_{B_{r}(u')}|u-u'|^2K'_{f_j}(u,u')\; du &\leq \int_{B_{r}(u')}|u-u'|^2\overline{K}'_{f_j}(u,u')\; du \lesssim (1+|u'|)^{\gamma+2s}r^{2-2s}, \\
 \int_{\mathbb{R}^d/B_{r}(u')}K'_{f_j}(u,u')\; du &\lesssim (1+|u'|)^{\gamma+2s}r^{-2s}.
\end{align*}
(ii) For any $1<r<+\infty$,
\begin{align*}  
    \int_{B_{r}(u')}|u-u'|^2K'_{f_j}(u,u')\; du &\leq \int_{B_{r}(u')}|u-u'|^2\overline{K}'_{f_j}(u,u')\; du \lesssim (1+|u'|)^{\gamma+2s}r^{\gamma+3},\\
    \int_{\mathbb{R}^d/B_{r}(u')}K'_{f_j}(u,u')\; du &\lesssim (1+|u'|)^{\gamma+2s}r^{\gamma}.
\end{align*}
Here,  $B_r(u') =\{x\in \mathbb{R}^d | \;|u'-x|<r \}$ and $K'_{f_j}(u,u')$, $\overline{K}'_{f_j}(u,u')$ are given in \eqref{Kf_M_2}, \eqref{K'f_j_M_2}, respectively. 
\end{lemma}
\begin{proof} 
In Lemma \ref{Kf_est_I}, we use point $R$ defined in  \eqref{Pre_R_M_2} instead of the point $P$. Also, we use the following triangle inequalities,
\begin{align} 
  &|v_*-u'| \leq |v_*-R|+\frac{m_i-m_j}{2m_j}|u-u'|,\label{Kf_inequal_M_2}\\
  &|v_*-R| \leq |v_*-u'|+\frac{m_i-m_j}{2m_j}|u-u'|, \label{Kf_inequal_M_2'}
\end{align} instead of \eqref{tri_inequ_Kf_est_I}, \eqref{tri_inequ_Kf_est_I'}. Here, the point $R$ is also defined in \eqref{Pre_R_M_2}. The proof of this Lemma is almost similar to Lemma \ref{Kf_est_I} and we omit the details.
\end{proof} \vspace{3mm}

\begin{lemma} ($Q^s$ estimate for $m_i<m_j$) \label{Q_est_M_1}
Assume $\gamma<0$ that and $\gamma+2s \in [0,2]$(moderately soft potentials) and $E_0,M_0<+\infty$ in \eqref{hydro_M}. Let $f_i$, $f_j$ be the density of particles of mass $m_i$, $m_j$ and $m_i<m_j$. Let $\psi$ be a bounded, $C^2$ function. The $Q^{s}_{ji}(f_j,\psi)(v)$ in \eqref{Qs_Kf_j_M_1} is written by
\begin{align*}
Q^{s}_{ji}(f_j,f_i)(t,v) &= \int K_{f_j}(v,v')(f_i(v')-f_i(v)) \;dv'. 
\end{align*}\\
\noindent If $\psi$ satisfies $\|\psi\|_{L^{\infty}} \leq \text{max} \{\|\nabla^2 \psi\|_{L^\infty}, \|\nabla\psi\|_{L^\infty} \}$, then
\begin{align} \label{Q_est_1_M_1}
     |Q_{ji}^{s}(f_j,\psi)(v)| 
    \lesssim
  \|\psi\|_{L^{\infty}}^{1-s}(\text{max}\{\|\nabla^2 \psi\|_{L^\infty}, \|\nabla \psi\|_{L^\infty} \})^{s}(1+|v|)^{\gamma+2s}. \\ \notag
\end{align} 
Or else, if $\psi$ satisfies $\|\psi\|_{L^{\infty}} > \text{max} \{\|\nabla^2 \psi\|_{L^\infty}, \|\nabla\psi\|_{L^\infty} \}$, then
\begin{align*}
     |Q^{s}_{ji}(f,\psi)(v)| 
    \lesssim
    \|\psi\|_{L^{\infty}}^{1+\gamma/3}(\text{max}\{\|\nabla^2 \psi\|_{L^\infty}, \|\nabla \psi\|_{L^\infty} \})^{-\gamma/3}(1+|v|)^{\gamma+2s}.
\end{align*}
\end{lemma}
\begin{proof}  
The proof of this Lemma is almost similar to Lemma \ref{Q_est_I}, and we apply Lemma \ref{Kf_j_estimate_M_1} instead of Lemma \ref{Kf_est_I}. We omit details.
\end{proof} \vspace{3mm}

\begin{lemma} ($Q^s$ estimate for $m_i > m_j$) \label{Q_est_M_2} Assume that $\gamma<0$ and $\gamma+2s \in [0,2]$(moderately soft potentials) and $E_0,M_0<+\infty$ in \eqref{hydro_M}. Let $f_i$, $f_j$ be the density of particles of mass $m_i$, $m_j$ and $m_i>m_j.$ Let $\psi$ be a bounded, $C^2$ function. The $Q^{s}_{ji}(f_j,\psi)(v)$ in \eqref{Qs_ji_Cal_M_2} is written by
\begin{align*}
 Q_{ji}^{s}(f_j,f_i)(v) \doteq \int K'_{f_j}(u,u')\delta(u'-v)(f_i(u)-f_i(u')) \;dudu'= \int K'_{f_j}(u,v)(f_i(u)-f_i(v))\;du.
\end{align*} \\
\noindent If $\psi$ satisfies $\|\psi\|_{L^{\infty}} \leq \text{max} \{\|\nabla^2 \psi\|_{L^\infty}, \|\nabla\psi\|_{L^\infty} \}$, 
\begin{align} \label{Q_est_1_M_2}
     |Q_{ji}^{s}(f_j,\psi)(v)| 
    \lesssim
  \|\psi\|_{L^{\infty}}^{1-s}(\text{max}\{\|\nabla^2 \psi\|_{L^\infty}, \|\nabla \psi\|_{L^\infty} \})^{s}(1+|v|)^{\gamma+2s},\\ \notag  
\end{align} 
Or else, if $\psi$ satisfy $\|\psi\|_{L^{\infty}} > \text{max} \{\|\nabla^2 \psi\|_{L^\infty}, \|\nabla\psi\|_{L^\infty} \}$, 
\begin{align*}
     |Q^{s}_{ji}(f,\psi)(v)| 
    \lesssim
    \|\psi\|_{L^{\infty}}^{1+\gamma/3}(\text{max}\{\|\nabla^2 \psi\|_{L^\infty}, \|\nabla \psi\|_{L^\infty} \})^{-\gamma/3}(1+|v|)^{\gamma+2s}.
\end{align*}
\end{lemma}
\begin{proof} 
The proof of this Lemma is almost similar to Lemma \ref{Q_est_I}, and we apply Lemma \ref{Kf_j_estimate_M_2} instead of Lemma \ref{Kf_est_I}. We omit details.
\end{proof} \vspace{3mm}

\subsubsection{The nonsingular part for mixture model}
We prove the cancellation lemma in mixture model.
If $m_i>m_j$, the function T is not well-defined and not one to one function in \eqref{T_M_1}, \eqref{T'_M_1}, and $|\frac{dv_*'}{dv_*}|$ can be zero in \eqref{Jacobian_CC_M_1}. Thus, we use the weak form of the collision operator in Lemma \ref{CC_M_2}.
\vspace{2mm}

\begin{lemma} (Cancellation lemma for $m_i < m_j$) \label{CC_M_1}
Suppose that $B_{ji}(|v-v_*|,\cos a)=|v-v_*|^{\gamma}b_{ji}(\cos a), \;a \in [0,\pi]$
and $\gamma>-d$ in \eqref{B_nc_M'}. Let $f_i$, $f_j$ be the density of particles of mass $m_i$, $m_j$ and $m_i<m_j.$ Then the $Q_{ji}^{ns}(f_j,f_i)(v)$ in \eqref{split_Q_M_1} is written by  
\begin{align}  \label{Qns_CC_M_1} 
\begin{split}
    Q_{ji}^{ns}(f_j, f_i)(v) &= f_i(v)\int B_{ji}(|v-v_*|, \cos a)(f_j(v'_*)-f_j(v_*)) \; d\sigma dv_*
    \\&= f_i(v) \int f_j(v_*)S_{ji}(|v-v_*|) \; dv_*,
\end{split}
\end{align}
where  
\begin{align*}
    S_{ji}(|v-v_*|)
    &= |\mathbb{S}^{d-2}||v-v_*|^{\gamma}\int_0^\pi b_{ji}(\cos w)\sin^{d-2}w [(\frac{m_i}{m_i+m_j}\cos a + \frac{m_j}{m_i+m_j}\cos A)^{-d-\gamma}-1] \; dw.
\end{align*}
Here, $w \doteq \sin^{-1}(\frac{m_i}{m_j}\sin a)+a \doteq A+a(\;0<A<\frac{\pi}{2}, \; 0<w,a<\pi \;).$
Moreover, $Q_{ji}^{ns}(f_j,f_i)(v)$ is well-defined and positive.
\end{lemma}
\begin{proof} 
Recall \eqref{v'_M'},
\begin{align*}
       v' = \frac{m_i}{m_i+m_j}v+\frac{m_j}{m_i+m_j}v_*+\frac{m_j}{m_i+m_j}|v-v_*|
    \sigma, \notag \\
    v_*' =\frac{m_i}{m_i+m_j}v+\frac{m_j}{m_i+m_j}v_*-\frac{m_i}{m_i+m_j}|v-v_*|\sigma.
\end{align*}
Let us write $k,{'k}$ and the angles $a,b$ as 
\begin{figure} [t]
\centering
\includegraphics[width=4cm]{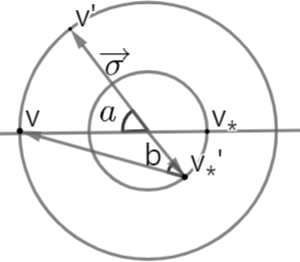}
\caption{The relation of $v,v_*$ and $v'_*$ \label{CC_M_1_F}}
\end{figure}
\begin{align*}
    k \cdot \sigma \doteq \frac{v-v_*}{|v-v_*|}\cdot \sigma = \cos a, \quad
    'k \cdot \sigma \doteq \frac{v-v_*'}{|v-v_*'|}\cdot \sigma = \cos b,
    \quad
    a,b \in (0,\pi). \text{\quad(See Figure \ref{CC_M_1_F}.)}
\end{align*}
\noindent Using the law of sine in Figure \ref{CC_M_1_F}, we get
\begin{align} \label{sine_law_M_1}
    \frac{\frac{m_j}{m_i+m_j}|v-v_*|}{\sin b} = \frac{|v-v_*'|}{\sin a} =\frac{\frac{m_i}{m_i+m_j}|v-v_*|}{\sin(a-b)}. 
\end{align}
Here, $a=b=0$, $a=b=\pi$, and if $a,b \in (0,\pi)$, then $a \neq b.$
From \eqref{sine_law_M_1}, we get $\sin(a-b)=\frac{m_i}{m_j}\sin b$. Since $0<\cos^{-1}(\frac{v-v_*}{|v-v_*|}\cdot \frac{v-v_*'}{|v-v_*'|})<\frac{\pi}{2}$, we infer $0<a-b<\frac{\pi}{2}$ and $a-b = \sin^{-1}(\frac{m_i}{m_j}\sin b)$. We define the map as 
\begin{align} 
    T(\cos b) = \cos(\sin^{-1}(\frac{m_i}{m_j}\sin b)+b) = \cos a, \label{T_M_1}
\end{align} then
\begin{align}
    T' = \frac{\sin a}{\sin b}(1+\frac{\frac{m_i}{m_j}\cos b}{\sqrt{1-(\frac{m_i}{m_j})^{2}\sin^{2}a}})>0 \label{T'_M_1} 
\end{align} since $m_i<m_j$ and $a,b \in (0,\pi).$ We define $\Psi_\sigma(v_*')=v_*$ where $v$ is fixed.
From \eqref{sine_law_M_1}, we get
\begin{align*}
    |v-v_*|=|v-\Psi_\sigma(v_*')|=\frac{\sin b}{\sin a}(\frac{m_j}{m_i+m_j})^{-1}|v-v_*'|=\frac{\sin b}{\sqrt{1-T^2(\cos b)}}(\frac{m_j}{m_i+m_j})^{-1}|v-v_*'|.
\end{align*} After renaming $v_*'$ to $v_*$ above, we obtain
\begin{align} \label{CC_M_replace}
    |v-\Psi_\sigma(v_*)|=\frac{\sin a}{\sqrt{1-T^2(\cos a)}}(\frac{m_j}{m_i+m_j})^{-1}|v-v_*|.
\end{align}
For the first term of  $Q^{ns}_{ji}(f_j,f_i)$, we perform change of variables $v_*\rightarrow v_*'$. From \eqref{v'_M'}, Jacobian determinant is obtained by 
\begin{align} \label{Jacobian_CC_M_1}
    |\frac{dv_*'}{dv_*}|=(\frac{m_j}{m_i+m_j})^d(1+\frac{m_i}{m_j}\frac{v-v_*}{|v-v_*|}\cdot \sigma)>0, 
\end{align} and this is always positive since $m_i<m_j$ We rename $v_*'$ to $v_*$ and replace $|v-\Psi_\sigma(v_*)|$ by \eqref{CC_M_replace}, then we get \eqref{rename_CC_M_1}. Thus, 
\begin{align}
     &\quad \int B_{ji}(|v-v_*|,\cos a)f_j(v_*') \;d\sigma dv_*  \label{trans_CC_M_1} \\ 
     &= (\frac{m_j}{m_i+m_j})^{-d}\int  B_{ji}(|v-v_*|,\cos a)f_j(v_*')(1+\frac{m_i}{m_j}\cos a)^{-1} \;d\sigma dv_*'  \notag \\ 
     &= (\frac{m_j}{m_i+m_j})^{-d}\int  B_{ji}(|v-\Psi_\sigma(v_*')|,T(\cos b)f_j(v_*')(1+\frac{m_i}{m_j}T(\cos b))^{-1} \;d\sigma dv_*' \notag \\
     &= (\frac{m_j}{m_i+m_j})^{-d}\int  B_{ji}(|v-\Psi_\sigma(v_*)|,T(\cos a)f_j(v_*)(1+\frac{m_i}{m_j}T(\cos a))^{-1} \;d\sigma dv_* \notag \\
     &= (\frac{m_j}{m_i+m_j})^{-d}\int B_{ji}((\frac{m_j}{m_i+m_j})^{-1}\frac{\sin a}{\sqrt{1-T^2(\cos a)}}|v-v_*|
   ,T(\cos a))f_j(v_*)(1+\frac{m_i}{m_j}T(\cos a))^{-1} \;d\sigma dv_* \label{rename_CC_M_1}.
\end{align}
\noindent  Let $w \doteq \sin^{-1}(\frac{m_i}{m_j}\sin a)+a\doteq A+a$, \;and\; $0 < A < \frac{\pi}{2}$, and $0 < w, a < \pi.$ Note that $T(\cos a) = \cos w$. 
The Jacobian determinant $|\frac{d w}{d a}|=1+\frac{\frac{m_i}{m_j}\cos a}{\sqrt{1-(\frac{m_i}{m_j})^{2}\sin^{2}a}}$ is positive since $m_i<m_j$. We observe
\begin{equation} \label{using}
\begin{split}
    &\frac{\sin a}{\sin w} = (\frac{m_i}{m_j}\cos a + \cos A)^{-1}, \quad
    1+\frac{m_i}{m_j}\cos w =\cos A\;(\frac{m_i}{m_j}\cos a + \cos A),\\
    &|\frac{dw}{da}|=
    1+\frac{\frac{m_i}{m_j}\cos a}{\sqrt{1-(\frac{m_i}{m_j})^{2}\sin^{2}a}} =\frac{1}{\cos A}  (\frac{m_i}{m_j}\cos a + \cos A).
\end{split}
\end{equation}
By changing of variables $\sigma \rightarrow a$ and $a \rightarrow w$, 
we obtain
\begin{align}
    &\quad \eqref{rename_CC_M_1} \notag \\
    &=|\mathbb{S}^{d-2}|(\frac{m_j}{m_i+m_j})^{-d}\int \int_0^\pi \sin^{d-2}a \; B_{ji}(\;\frac{\sin a}{\sin w}(\frac{m_j}{m_i+m_j})^{-1}|v-v_*|,\cos w)
    f_j(v_*)(1+\frac{m_i}{m_j}\cos w)^{-1}|\frac{da}{dw}|\; dwdv_*  \notag \\
    &= |\mathbb{S}^{d-2}|(\frac{m_j}{m_i+m_j})^{-d-\gamma}\int \int_0^\pi |v-v_*|^{\gamma} \frac{\sin^{\gamma+d-2}a}{\sin^{\gamma}w}(1+\frac{m_i}{m_j}\cos w)^{-1}b_{ji}(\cos w) f_j(v_*)|\frac{da}{dw}| \; dwdv_* \notag\\ 
     &=|\mathbb{S}^{d-2}|\int \int_0^\pi |v-v_*|^{\gamma}b_{ji}(\cos w)\sin^{d-2}w (\frac{m_i}{m_i+m_j}\cos a + \frac{m_j}{m_i+m_j}\cos A)^{-d-\gamma}f_j(v_*) \; dwdv_*, \notag
\end{align} 
where we used \eqref{using} in the last step. In conclusion, we get 
\begin{align*}
    \int B_{ji}(|v-v_*|, \cos a)(f_j(v'_*)-f_j(v_*)) \; d\sigma dv_*
    = \int f_j(v_*)S_{ji}(|v-v_*|) \; dv_*,
\end{align*}
where 
\begin{align*}
    S_{ji}(|v-v_*|)
    &= |\mathbb{S}^{d-2}||v-v_*|^{\gamma}\int_0^\pi b_{ji}(\cos w)\sin^{d-2}w [(\frac{m_i}{m_i+m_j}\cos a + \frac{m_j}{m_i+m_j}\cos A)^{-d-\gamma}-1] \; dw > 0.
\end{align*}
Using similar argument as Lemma \ref{CC_I}, we get
\begin{align*}
    (\frac{m_i}{m_i+m_j}\cos a + \frac{m_j}{m_i+m_j}\cos A)^{-d-\gamma}  &\geq \cos^{-d-\gamma}A \geq 1,
\end{align*}
and
\begin{align*}
     1-(\frac{m_i}{m_i+m_j}\cos a + \frac{m_j}{m_i+m_j}\cos A)^{d+\gamma}
     &\lesssim \sin^{2}\frac{a}{2}+\sin^{2}\frac{A}{2}
 \leq 2\sin^{2}\frac{w}{2}
\end{align*} under the condition $\gamma \geq -d.$ This implies that $S_{ji}(|v-v_*|)$ is positive and well-defined.
\end{proof} \vspace{3mm}

\begin{lemma} (Cancellation lemma for $m_i > m_j$) \label{CC_M_2}
Suppose that $B_{ji}(|v-v_*|,\cos a)=|v-v_*|^{\gamma}b_{ji}(\cos a), \;a \in [0,\pi]$
and $\gamma>-d$ in \eqref{B_nc_M'}. Let $f_i$, $f_j$ be the density of particles of mass $m_i$, $m_j$ and $m_i>m_j.$ Then the $Q_{ji}^{ns}(f_j,f_i)(v)$ in \eqref{Qs_Qns_M_2} is written by 
\begin{align} \label{Qns_CC_M_2} 
\begin{split}
    Q^{ns}_{ji}(f_j, f_i)(v) &= \int f_j(v_*)\int B_{ji}(|u-v_*|, \cos a)(\delta(u'-v)f_i(u')-\delta(u-v)f_i(u)) \; d\sigma du dv_*  \\&=\int f_j(v_*) \int \delta(u-v)f_i(u)S'_{ji}(|u-v_*|) \; du
    dv_*
    =f_i(v) \int f_j(v_*)S'_{ji}(|v-v_*|) \;
    dv_*,
\end{split}
\end{align}
where  
\begin{align*}
    S'_{ji}(|v-v_*|)
    &= |\mathbb{S}^{d-2}||v-v_*|^{\gamma}\int_0^\pi b_{ji}(\cos w)\sin^{d-2}w [(\frac{m_j}{m_i+m_j}\cos a + \frac{m_i}{m_i+m_j}\cos A)^{-d-\gamma}-1] \; dw.
\end{align*}
Here, $w \doteq \sin^{-1}(\frac{m_j}{m_i}\sin a)+a \doteq A+a(\; 0<A<\frac{\pi}{2}, \; 0<w,a<\pi \;).$
Moreover, $Q_{ji}^{ns}(f_j,f_i)(v)$ is well-defined and positive.
\end{lemma}
\begin{proof} 
We rewrite $v$ to $u$ and $v'$ instead of $u'$ in \eqref{v'_M'}, then  
\begin{align} \label{u_M}
\begin{split}
     u'= \frac{m_i}{m_i+m_j}u+\frac{m_j}{m_i+m_j}v_*+\frac{m_j}{m_i+m_j}|u-v_*|\sigma,\\
    v_*'=\frac{m_i}{m_i+m_j}u+\frac{m_j}{m_i+m_j}v_*-\frac{m_i}{m_i+m_j}|u-v_*|\sigma.
\end{split}
\end{align}
Let us write $k,{'k}$ and the angles $a,b$ as 
\begin{figure} [t]
\centering
\includegraphics[width=4cm]{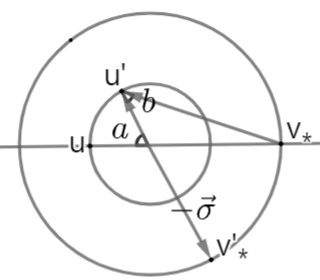}
\caption{The relation of $u,v_*$ and $u'$ \label{CC_M_2_F}}
\end{figure} 
\begin{align*}
    k \cdot {\sigma} \doteq \frac{u-v_*}{|u-v_*|}\cdot {\sigma} = \cos a, \quad
    'k \cdot {\sigma} \doteq \frac{u'-v_*}{|u'-v_*|}\cdot {\sigma} = \cos b,
    \quad
    a,b \in (0,\pi). \text{\quad(See Figure \ref{CC_M_2_F}.)}
\end{align*} 
Using the law of sine in Figure \ref{CC_M_2_F}, we get
\begin{align} \label{sine_law_M_2}
    \frac{\frac{m_i}{m_i+m_j}|u-v_*|}{\sin b} = \frac{|u'-v_*|}{\sin a} =\frac{\frac{m_j}{m_i+m_j}|u-v_*|}{\sin(a-b)}. 
\end{align} Here, $a=b=0$, $a=b=\pi$, and if $a,b \in (0,\pi)$, then $a \neq b.$ From \eqref{sine_law_M_2}, we get $\sin(a-b)=\frac{m_j}{m_i}\sin b$. Since $0<\cos^{-1}(\frac{u-v_*}{|u-v_*|}\cdot \frac{u'-v_*'}{|u'-v_*'|})<\frac{\pi}{2}$, we infer $0<a-b<\frac{\pi}{2}$ and $a-b = \sin^{-1}(\frac{m_j}{m_i}\sin b).$ We define the map $T$ as  
\begin{align*}
    &T(\cos b) = \cos(\sin^{-1}(\frac{m_j}{m_i}\sin b)+b) = \cos a,
\end{align*} then 
\begin{align*}
    &T' = \frac{\sin a}{\sin b}(1+\frac{\frac{m_j}{m_i}\cos b}{\sqrt{1-(\frac{m_j}{m_i})^{2}\sin^{2}a}})>0 
\end{align*} since $m_i>m_j$ and $a,b \in (0,\pi).$ We define $\Psi_\sigma(u')=u$ where $v_*$ is fixed. 
For the first term of $Q^{ns}_{ji}(f_j,f_i)$, we perform change of variables $u \rightarrow u'$. From \eqref{u_M}, Jacobian determinant is obtained by 
\begin{align*} 
    |\frac{du'}{du}|=(\frac{m_i}{m_i+m_j})^d(1+\frac{m_j}{m_i}\frac{u-v_*}{|u-v_*|}\cdot \sigma)>0
\end{align*} and this is always positive since $m_i>m_j.$ Next, we rename $u'$ to $u$. Using similar argument as Lemma \ref{CC_M_1}, then we have that 
\begin{align} 
    &\quad \int B(|u-v_*|,\cos a) \delta(u'-v)f_i(u') \;d\sigma du \label{B_CC_M_2} \\
    &=(\frac{m_i}{m_i+m_j})^{-d}\int B(\frac{\sin a}{\sqrt{1-(T(\cos a))^2}}(\frac{m_i}{m_i+m_j})^{-1}|u-v_*|,\;T(\cos a)) \delta(u-v)f_i(u)(1+\frac{m_j}{m_i}T(\cos a))^{-1} \;d\sigma du \notag \\ 
    &=|\mathbb{S}^{d-2}|\int \int_0^\pi |u-v_*|^{\gamma}b_{ji}(\cos w)\sin^{d-2}w (\frac{m_j}{m_i+m_j}\cos a + \frac{m_i}{m_i+m_j}\cos A)^{-d-\gamma}\delta(u-v)f_i(u) \; dwdu \notag 
\end{align}
\noindent and get $S'_{ji}(|v-v_*|)$ in the statement. Also, we obtain that $Q^{ns}_{ji}(f_j,f_i)$ is well-defined and positive.
\end{proof}

\begin{remark}
Note that if $m_i=m_j$ and the collision occur only between identical particles, then $A=a=\frac{w}{2}$ and $ S_{ji}(|v-v_*|)= S'_{ji}(|v-v_*|)= S(|v-v_*|)$, where
\begin{equation} \label{S_elastic'}
      S(|v-v_*|)=|\mathbb{S}^{d-2}||v-v_*|^{\gamma}\int_0^\pi b(\cos w)\sin^{d-2}w(\cos^{-d-\gamma}\frac{w}{2}-1) \;dw.
\end{equation} The above equation, \eqref{S_elastic'} is same $S(|v-v_*|)$ in Lemma 1 in \cite{AD2000}. \\
\end{remark} 

\subsubsection{Lowerbound for mixture model} 
We prove that $f_i(t,v)$ in \eqref{BT_M} is strictly positive in mixture model. 

\begin{lemma} (Positivity for mixture model) \label{pos_nc_M}
Suppose that $f_i(t,v)$ in \eqref{BT_M} is $C^{\infty}$ function. Then $f_i(t,v)>0$ on $(0,+\infty)\times \mathbb{R}^{d}.$ 
\end{lemma}
\begin{proof}
The $f_i(t,v)$ is a density of particles of mass $m_i$ for $1 \leq i \leq N$. Assume $m_1<m_2 \cdots < m_N$ WLOG. \vspace{3mm}

\noindent When $m_i>m_j$, recall \eqref{Qs_M_2} and \eqref{Qns_M_2}, 
\begin{align}
    Q_{ji}(f_j,f_i)(v) &= Q^{s}_{ji}(f_j,f_i)(v) + Q^{ns}_{ji}(f_j,f_i)(v) \notag \\ 
     &=\int B_{ji}(|u-v_*|, \cos a)f_j(v_*)\delta(u'-v)(f_i(u)-f_i(u'))\; d\sigma dv_*du \notag \\ 
     &\quad+\int f_j(v_*)\int B_{ji}(|u-v_*|, \cos a)(\delta(u'-v)f_i(u')-\delta(u-v)f_i(u)) \; d\sigma du dv_*. \label{pos_Q_ji_1}
\end{align}
\noindent Taking $f_i(u)$ instead of $f_i(u')$ in \eqref{B_CC_M_2}, then the first term of $Q_{ji}^{s}(f_j,f_i)(v)$ is expressed in $\overline{S}'_{ji}(|v-v_*|,w)$ below. 
Similar as Lemma \ref{pos_nc_I}, we get
\begin{align}
     \eqref{pos_Q_ji_1}
     =\int f_j(v_*) \int \overline{S}'_{ji}(|v-v_*|,w)(f_i('v)-f_i(v)) \;dwdv_* 
+ f_i(v)\int f_j(v_*)S'_{ji}(|v-v_*|)  \;dv_*,   \label{S'_pos_nc_M}
\end{align}
where $\overline{S}'_{ji}(|v-v_*|,w)$ is 
\begin{equation*}
    \overline{S}'_{ji}(|v-v_*|,w)=|\mathbb{S}^{d-2}||v-v_*|^{\gamma}b(\cos w)\sin^{d-2}w(\frac{m_j}{m_i+m_j}\cos a + \frac{m_i}{m_i+m_j}\cos A)^{-d-\gamma}.
\end{equation*}
Here, $S'_{ji}(|v-v_*|)$ is in Lemma \ref{CC_M_2} and $'v$ is pre-velocity and defined 
\begin{align*} 
    v &= \frac{m_i}{m_i+m_j}{'v}+\frac{m_j}{m_i+m_j}{'v_*}+\frac{m_j}{m_i+m_j}|'v-{'v_*}|\sigma, \notag \\
    v_* &=\frac{m_i}{m_i+m_j}{'v}+\frac{m_j}{m_i+m_j}{'v_*}-\frac{m_i}{m_i+m_j}|'v-{'v_*}|\sigma.
\end{align*} \vspace{3mm}

\noindent When $m_i \leq m_j,$ recall \eqref{Qs_M_1} and \eqref{Qns_M_2}, 
\begin{align}
    Q_{ji}(f_j,f_i)(t,v) &= Q^{s}_{ji}(f_j,f_i)(t,v) + Q^{ns}_{ji}(f_j,f_i)(t,v) \notag \\
    &=  \int B_{ji}(|v-v_*|,\cos a)f_j(v_*')(f_i(v')-f_i(v))\;d\sigma dv_* +  f_i(v)\int B_{ji}(|v-v_*|,\cos a)(f_j(v_*')-f_j(v_*)) \; d\sigma dv_* \notag \\ 
    &=  \int B_{ji}(|v-v_*|,\cos a)f_j(v_*')(f_i(v')-f_i(v))\;d\sigma dv_* +  f_i(v)\int f_j(v_*)S_{ji}(|v-v_*|)  \;dv_* \label{S_pos_nc_M},
\end{align} where $S_{ji}(|v-v_*|)$ is in Lemma \ref{CC_M_1}.\vspace{3mm}
\noindent Combining \eqref{S'_pos_nc_M} and \eqref{S_pos_nc_M} , we obtain
\begin{align}
    \partial_t f_i (t,v) &= \sum_{j=1}^{N} Q_{ji}(f_j,f_i) (t,v) 
    =\sum_{j=1}^{i-1}Q_{ji}(f_j,f_i) (t,v) + \sum_{j=i}^{N} Q_{ji}(f_j,f_i) (t,v) \notag \\ 
    &=\sum_{j=1}^{i-1}(Q^{s}_{ji}(f_j,f_i) (t,v)+Q^{ns}_{ji}(f_j,f_i) (t,v)) + \sum_{j=i}^{N} (Q^{s}_{ji}(f_j,f_i) (t,v)+Q^{ns}_{ji}(f_j,f_i) (t,v)) \notag \\ 
    &=\sum_{j=1}^{i-1}(\int f_j(v_*) \int \overline{S}'_{ji}(|v-v_*|,w)(f_i('v)-f_i(v)) \;dwdv_* 
+ f_i(v)\int f_j(v_*)S'_{ji}(|v-v_*|)  \;dv_*) \notag \\ 
&\quad + \sum_{j=i}^{N}(\int B_{ji}(|v-v_*|,\cos a)f_j(v_*')(f_i(v')-f_i(v))\;d\sigma dv_* + f_i(v)\int f_j(v_*)S_{ji}(|v-v_*|)  \;dv_*). \label{last_pos_nc_M}
\end{align}
\noindent 
Now, we can apply contradiction argument which is used in p.103 in \cite{V2002}. Assume that $f(t,v)$ is zero at $(t_0,v_0)$. Then we get
\begin{align} \label{mix_pos_1}
    f_i(v)\int f_j(v_*)S'_{ji}(|v-v_*|)  \;dv_* = 0, \quad \int f_j(v_*) \int \overline{S}'_{ji}(|v-v_*|,w)(f_i('v)-f_i(v)) \;dwdv_*  \geq 0
\end{align}   
at $(t,v)=(t_0,v_0)$ for $1 \leq j \leq i-1$ and 
\begin{align} \label{mix_pos_2}
    f_i(v)\int f_j(v_*)S_{ji}(|v-v_*|)  \;dv_* = 0, \quad \int B_{ji}(|v-v_*|,\cos a)f_j(v_*')(f_i(v')-f_i(v))\;d\sigma dv_* \geq 0
\end{align}
at $(t,v)=(t_0,v_0)$ for $i \leq j \leq N$. 
Applying \eqref{mix_pos_1} and \eqref{mix_pos_2} to  \eqref{last_pos_nc_M}, we get 
\begin{align*}
\partial_t f_{i}(t,v) \geq
    \int B_{ii}(|v-v_*|,\cos a)f_i(v_*')(f_i(v')-f_i(v))\;d\sigma dv_*.
\end{align*}  
Since $f_i(t_0,v) \geq 0$ for all $v$ and $f_i(t_0,v_0)=0$, we get
\begin{align*}
    \partial_t f_{i}(t,v)=0, \quad f_i(v')-f_i(v) \geq 0
\end{align*}
at $(t,v)=(t_0,v_0).$ Then $f(t_0,v')=f(t_0,v_0)=0$ for all $v' \in \mathbb{R}^d$ and it is contradiction. 
\end{proof} \vspace{3mm}

The spreading lemma in elastic (mono-species) model has been proved in Lemma 3.4 in \cite{IM2020} and we extend the result to mixture model.

\begin{lemma} (Spreading lemma for mixture model) \label{spreading_M} Consider $T_0 \in (0,1).$
Suppose that $\gamma<0$ and $\gamma+2s \in [0,2]$ (moderately soft potentials) and $f_i(t,v)$ is the solution in \eqref{BT_M}. 
If $f_i(t,v) \geq l$ on $t \in [0,T_0],\; v\in B_R(0)$ for some $l > 0, R > 1,$ then there exists constant $C>0$ depending on $d, s, M_0, E_0$ such that
\begin{align*}
    f_i(t,v) \geq C\; min(t, R^{-\gamma}\epsilon^{2s})\; \epsilon^q R^{d+r}l^2, \quad \forall t\in[0,T_0],\; \forall v\in B_{\sqrt{2}(1-\epsilon)R}(0)
\end{align*}
for any $\epsilon \in (0, 1-\frac{1}{\sqrt{2}})$ that satisfied with $\epsilon^{q}R^{d+r}l < 1/2$ and $R\epsilon < 1.$
\end{lemma}
\begin{proof}
We will obtain this lemma using Lemma \ref{Q_est_M_1}, \ref{Q_est_M_2}, \ref{CC_M_1}, and \ref{CC_M_2}. \vspace{3mm}

\noindent First, we refer to \eqref{f_condition} and derive $\phi_{R,\epsilon}(u)$,
\begin{align*}
    \phi_{R,\epsilon}(u) =
     \left\{ \begin{array}{rcl}1 & \mbox{for} \; |u| \leq \sqrt{2}R(1-\epsilon)\\
0 & \mbox{\quad for} \;|u|> \sqrt{2}R(1-\epsilon/2)
\end{array}\right.
\end{align*}
such that \eqref{phi condition}.
By \eqref{Q_est_1_M_1} and \eqref{Q_est_1_M_2}, there exists constant $c>0$ such that 
\begin{align}
  |Q_{ji}^{s}(f_j,\phi_{R,\epsilon})(v)| 
    &\lesssim
  \|\phi_{R,\epsilon}\|_{L^{\infty}}^{1-s}(\text{max}\{\|\nabla^2 \phi_{R,\epsilon}\|_{L^\infty}, \|\nabla \phi_{R,\epsilon}\|_{L^\infty} \})^{s}(1+|v|)^{\gamma+2s} \notag \\
  &= \|\phi_{R,\epsilon}\|_{L^{\infty}}^{1-s}\|\nabla^2 \phi_{R,\epsilon}\|_{L^\infty}^{s}(1+|v|)^{\gamma+2s}
  \leq \frac{c}{N}R^\gamma \epsilon^{-2s} \label{Qs_est_spreading_M}
\end{align} for $1 \leq j \leq N$. \vspace{4mm}

\noindent We want to prove $f_i(t,v)\geq \tilde{l}(t)\phi_{R,\epsilon}(v)$ where $\tilde{l}(t)$ is in \eqref{l_spreading_I}. If the inequality was not true, there exists $(t_0,v_0)\in [0,T_0]\times \mbox{supp}\;\phi_{R,\epsilon}$ such that $f_i(t_0,v_0)=\tilde{l}(t_0)\phi_{R,\epsilon}(v_0)$ and $\partial_t(f_i-\tilde{l}(t)\phi_{R,\epsilon}(v))\leq 0$ at $(t_0,v_0).$ Using the fact that $Q^{ns}_{ji}(f_j,f_i)(v)>0$ and  \eqref{Qs_est_spreading_M}, we have that
\begin{align}
    \tilde{l}'(t_0) \;\; &\geq \tilde{l}'(t_0)\phi_{R,\epsilon}(v_0) \geq \partial_t f_i(t_0,v_0) \notag \\
     &\geq \sum_{j=1}^{N} Q_{ji}^{s}(f_j,f_i) (t_0,v_0)
     = \sum_{j=1}^{N} Q_{ji}^{s}(f_j,f_i-\tilde{l}\phi_{R,\epsilon}) (t_0,v_0)
    +\tilde{l}(t_0)\sum_{j=1}^{N} Q_{ji}^{s}(f_j,\phi_{R,\epsilon}) \notag \\
    &\geq \sum_{j=1}^{N} Q_{ji}^{s}(f_j,f_i-\tilde{l}\phi_{R,\epsilon}) (t_0,v_0)- cR^{\gamma}\epsilon^{-2s}\tilde{l}(t_0)   \label{2_spreading_M}. 
\end{align} 
\noindent Restricting $v'$ and $v_*'$ to $B_R(0),$ we get $f_i(v')-\tilde{l}(t_0)\phi_{R,\epsilon}(v') \geq l/2$ since assuming $\epsilon^{q}R^{d+r}l < 1/2$. Because the collision operator in mixture model is time reversible, \eqref{K_f_M_2_spreading_M} holds. For $1 \leq j \leq i-1$, we obtain  
\begin{align} \label{K_f_M_2_spreading_M}
    Q_{ji}^{s}(f_j,f_i-\tilde{l}\phi_{R,\epsilon})(t_0,v_0) &= \int K'_{f_j}(u,u')\delta(u'-v_0)(f_i(u)-\tilde{l}(t_0)\phi_{R,\epsilon}(u)) \;dudu' \notag \\
    &= \int K'_{f_j}('v,v_0)(f_i('v)-\tilde{l}(t_0)\phi_{R,\epsilon}('v) \;d'v \notag \\
    &= \int K'_{f_j}(v',v_0)(f_i(v')-\tilde{l}(t_0)\phi_{R,\epsilon}(v')) \;dv' >0,
\end{align} where $K'_{f_j}(u,u')$ is in \eqref{Kf_M_2}.  For $i \leq j \leq N$, we obtain 
\begin{align} \label{K_f_M_1_spreading_M}
    Q_{ji}^{s}(f_j,f_i-\tilde{l}\phi_{R,\epsilon})(t_0,v_0)=\int K_{f_j}(v_0,v')(f_i(v')-\tilde{l}(t_0)\phi_{R,\epsilon}(v'))\; dv' > 0,
\end{align}  where $K_{f_j}(v_0,v')$ is in \eqref{Kf_j_M_1}. 
Applying \eqref{K_f_M_2_spreading_M} and \eqref{K_f_M_1_spreading_M} to \eqref{2_spreading_M}, we get 
\begin{align}
     \tilde{l}'(t_0) \geq Q^{s}_{ii}(f_i, f_i-\tilde{l}\phi_{R,\epsilon}) (t_0,v_0) - cR^{\gamma}\epsilon^{-2s}\tilde{l}(t_0) \label{4_spreading_M}
\end{align} on $|v'|<R$ and $|v'_*|<R$.  \vspace{3mm}

\noindent Using the equation $\tilde{l}'(t_0)=\alpha\;\epsilon^q R^{d+r}l^2-cR^{\gamma}\epsilon^{-2s}\tilde{l}(t_0)$, we get  
\begin{align}
   \alpha\;\epsilon^q R^{d+r}l^2
   &\gtrsim \int \frac{f_i(v')-\tilde{l}(t_0)\phi_{R,\epsilon}(v')}{|v'-v_0|^{d+2s}}\int_{v_*'\in E_{Pv'}} |v_0-v_*'|^{\gamma+2s+1}f_i(v_*')\;\chi_{\{|v'|\leq R\}}\chi_{\{|v_*'|\leq R\}}\; dv_*'dv'\notag \\ 
   &\gtrsim  R^{-d-2s}l^2 \int \chi_{\{|v'|\leq R\}}\;\int_{v_*'\in E_{Pu'}}|v_0-v_*'|^{\gamma+2s+1} \chi_{\{|v_*'|\leq R\}} \, dv_*' dv' \notag \notag \\  
    &\geq C'\epsilon^q R^{d+\gamma}l^2  \label{mono_sp_ineq}
\end{align} on $v_0<\sqrt{2}R(1-\epsilon)$ for some constant $C'>0.$ There is only general collision term in RHS of \eqref{4_spreading_M}. So, we conclude \eqref{mono_sp_ineq} from Lemma 3.4 in \cite{IM2020} in elastic mono-species model. It is contradiction. 
\end{proof} \vspace{3mm}

\begin{proof} [Proof of Theorem \ref{1.2}]
For any $T_0 \in (0,1),$ we define 
\begin{align*} 
     &T_n = (1-\frac{1}{2^{n}})T_0, \\
      &\epsilon_n=\frac{1}{2^{n+1}}, \\
       &R_{n+1}=\sqrt{2}(1-\epsilon_n)R_n, \;\; R_0=1.
\end{align*}
By Lemma \ref{pos_nc_M}, there exists some constant $l_0 \in (0,1)$ such that $f_i(t,v) \geq l_0$ when $t \in [T_1,T_0], v \in B_{R_0}(0)$ under the condition that $f_i(t,v) \in C^{\infty}.$ For $v \in B_{R_n}(0)$, 
we assume that $f_i(t,v) \geq l_n$ when $t \in [T_{n+1},T_0], \; v \in B_{R_n}(0)$. We check $\epsilon_n^q R_n^{d+\gamma}l_n < 1/2$ and $R_n\epsilon_n<1$. 
Then, by Lemma \ref{spreading_M}, we get \;$f_i(t,v) \geq l_{n+1}$ when $t \in [T_{n+2},T_0], \; v \in B_{R_{n+1}}(0)$ and $l_{n+1} \geq K l_n^{2}$ for some constant $K$ depending on $d, s, M_0, E_0$ and $t.$ Now, similar as Proposition 3.6 in \cite{IM2020}, we get the Gaussian lowerbound.
\end{proof} \vspace{2mm}

\subsection{Cutoff collision kernel}
For $v \in \mathbb{R}^3$, the collision operator is written by
\begin{align*} 
    Q_{ji}(f_j, f_i) &(t,v) = \int B_{ji}(v-v_*, \sigma)(f_j(v_*')f_i(v')-f_j(v_*)f_i(v)) \; dn dv_*, \\  &\text{where} \; \cos\theta = \langle\frac{v-v_*}{|v-v_*|}, n\rangle, \quad \theta \in [0, \frac{\pi}{2}].
\end{align*} Here, $v'$ and $v_*'$ are given by \eqref{v'_M}. We define 
\begin{align*}
     B_{ji}(|v-v_*|,\theta) = h_{ji}(\theta)|v-v_*|^\gamma ,\quad \int_0^{\pi/2}h_{ji}(\theta) \;d\theta < +\infty
\end{align*} for hard potentials, $\gamma \in [0,1].$\\
We split  $Q_{ji}(f_j,f_i)$ into the gain term $Q_{ji}^{+}(f_j,f_i)$ and the loss term $Q_{ji}^{-}(f_j,f_i),$
\begin{align*} 
    Q_{ji}(f_j,f_i)(v)&=\int B(|v-v_*|,\theta)f_j(v'_*)f_i(v')\; dndv_* - \int B(|v-v_*|,\theta)f_j(v_*)f_i(v)\; dndv_*\\
    &\doteq Q_{ji}^{+}(f_j,f_i)(v)-Q_{ji}^{-}(f_j,f_i)(v). \notag
\end{align*} 

\noindent The loss term is written by 
\begin{align*}
    Q_{ji}^{-}(f_j,f_i)(t,v) = f_i(v)\int B_{ji}(|v-v_*|,\theta)f_j(v_*) \; dndv_* \doteq f_i(t,v) L(f_j)(t,v),
\end{align*}
and $L(f_j)(t,v)$ is estimated by  
\begin{align} \label{L_f_upper_I}
    L(f_j)(t,v) = \int h_{ji}(\theta)|v-v_*|^{\gamma}f_j(v_*) \; dn dv_* \leq C_1(1+|v|^{\gamma})
\end{align}
for some constant, $C_1>0.$ Here, the constant $C_1$ depends on $M_0, E_0$ and $\gamma$.
 Using (2.11) and (2.12) in \cite{AO2022}, we get
\begin{align*}
    \sum_{k=1}^{N} \int f_k(v) \log f_k(v) \; dv \leq \sum_{k=1}^{N} \int f_k(0,v) \log f_k(0,v) \; dv = H_0
\end{align*} and 
\begin{align*}
     - \sum_{k=1}^{N} \int_{f_k \leq 1} f_k(v) \log f_k(v) \; dv \leq C'\; (\sum_{k=1}^{N} \int f(0,v)(1+|v|^2)\; dv)^{3/4}.
\end{align*} Now, we apply above inequalities to \eqref{flogf_M}. Thus,
\begin{align}
    \int_{f_j >1} f_j(v) \log f_j(v) \; dv &\leq
    \sum_{k=1}^{N} \int_{f_k >1} f_k(v) \log f_k(v) \; dv \notag \\
    &= \sum_{k=1}^{N} \int f_k(v) \log f_k(v) \; dv
     - \sum_{k=1}^{N} \int_{f_k \leq 1} f_k(v) \log f_k(v) \; dv  \label{flogf_M} \\ 
    &\leq \sum_{k=1}^{N} \int f_k(0,v) \log f_k(0,v) \; dv + C' (\sum_{k=1}^{N} \int f(0,v)(1+|v|^2)\; dv)^{3/4} \notag \\
      &\leq C_2 \label{C_2_log_M}
\end{align}
for some constants $C', C_2>0.$ Here, the constant $C_2$ depends on $M_0,E_0$ and $H_0$.
From \eqref{C_2_log_M}, we can apply Lemma 4 in \cite{L1983} and get
\begin{align} \label{L_f_lower_I}
      L(f_j)(t,v) \geq C_3 (1+|v|^{\gamma}) 
\end{align} for some constant $C_3>0.$ Here, the constant $C_3$ depends on $M_0, H_0, E_0$ and $\gamma$.
Recall the mixture Boltzmann equation \eqref{BT_M}, 
\begin{align} 
    \partial_t f_i (t,v) &= \sum_{j=1}^{N} Q_{ji}(f_j,f_i) (t,v) \notag = \sum_{j=1}^{N}(Q_{ji}^{+}(f_j,f_i) (t,v)-Q_{ji}^{-}(f_j,f_i) (t,v)) \notag \\
    &= \sum_{j=1}^{N} Q_{ji}^{+}(f_j,f_i) (t,v) - f_i(v) \sum_{j=1}^{N} L(f_j)(t,v) \label{BT_M_L}
\end{align} for $v \in \mathbb{R}^3, \; t \in \mathbb{R^{+}}.$ Let $f_i(t,v)$ be a density of particles of mass $m_i$ at time $t$ and velocity $v$ for $1 \leq i \leq N$ and assume $m_1<m_2< \cdots < m_N.$ From \eqref{BT_M_L}, we obtain
\begin{align} \label{duhamel_M}
    f_i(t,v) 
    &={G_{0}^{t}}(v) f_i(0,v) + \sum_{j=1}^{N} \int_0^t G_{\tau}^{t}(v)\;Q_{ji}^{+}(f_j(\tau,.),f_i(\tau,.)(v))\; d\tau, 
\end{align}
where 
\begin{align*}
    G_{t_1}^{t_2}(v) \doteq e^{-\sum_{j=1}^{N} {\int_{t_1}^{t_2}} L(f_j)(\tau, v)d\tau}.
\end{align*}
By \eqref{L_f_upper_I} and \eqref{L_f_lower_I}, there are some constants $C_4,\; C>0$ depending on $M_0, H_0, E_0$ and $\gamma$ such that
\begin{align*}
    C_4(1+|v|^{\gamma}) \leq \sum_{j=1}^{N} L(f_j)(t,v) \leq C(1+|v|^{\gamma}).
\end{align*} If $|v|<R$ for some $R>0,$ then we get
\begin{align} \label{bar_G}
    G_{t_1}^{t_2}(v) \geq e^{-C(t_2-t_1) (1+R^\gamma)}\doteq \overline{G}_{t_1}^{t_2}(R).
\end{align} \vspace{3mm}

\begin{proof} [Proof of Theorem \ref{1.3}]
From \eqref{duhamel_M}, we get inequality, 
\begin{align} \label{positive_M}
    f_i(t,v)  &= G_{0}^{t}(v) f_i(0,v) + \sum_{j=1}^{N} \int_0^t G_{\tau}^{t}(v)\;Q_{ji}^{+}(f_j(\tau,.),f_i(\tau,.)(v))\; d\tau \notag\\
    &\geq 
    \int_0^t G_{\tau}^{t}(v)\;Q_{ii}^{+}(f_i(\tau,.),f_i(\tau,.)(v))\; d\tau,
\end{align}
where $Q_{ii}^{+}(f_i,f_i)(v)$ is general collision operator of elastic mono-species particles. Using \eqref{bar_G} and \eqref{positive_M}, there are $\epsilon_0,\delta_0,R$ and $\Bar{v} \in \mathbb{R}^3$ such that 
\begin{align*}
    f_i(t_0,v) \geq  \epsilon_0 \quad \text{for} \quad |v-\Bar{v}|<\delta_0, \;\; |\Bar{v}|<R
\end{align*} for any positive $t_0$ from Lemma 3.1 in \cite{AB1996}.  For $|v|<\sqrt{2}\delta_0(1-\gamma_1),$ this inequality also holds
\begin{align*}
    f_i(t_0+t_1,v) 
    &\geq \int_{t_1}^{t_0+t_1} G_{\tau}^{t}(v)\;Q_{ii}^{+}(f_i(\tau,.),f_i(\tau,.)(v))\; d\tau \\
    &\geq  C \; t_1 e^{-2ct_1 2^{\frac{\gamma}{2}}\delta_0^{\gamma}} \delta_0^{3+\gamma}\gamma_1^{\frac{5}{2}}\epsilon_0^2
\end{align*}
from Lemma 3.2 in \cite{AB1996}. Similar as Theorem 1.1 in \cite{AB1996} and we get the Gaussian lowerbound. 
\end{proof}

\section{Acknowledgements}%CK thanks the anonymous referees of \cite{Kim11} for the insightful reviews on a possible connection between the analytic singularities (e.g. \cite{AM,Taylor1,Taylor2}) and the singularity of the Boltzmann, which have been greatly inspiring the authors to study the subject of the current paper. He also thanks Laurent Desvillettes for his insightful comments on the subject in the occasion of Summer Workshop on Kinetic Theory and Gas Dynamics (organized by Tai-Ping Liu) at Stanford University June/2010. 
GA and DL are supported by the National Research Foundation of Korea(NRF) grant funded by the Korea government(MSIT)(No. NRF-2019R1C1C1010915). 

\bibliographystyle{abbrv}
\nocite{*}
\bibliography{reference.bib}

\begin{thebibliography}{10}

\bibitem{AD2000}
R.~Alexandre, L.~Desvillettes, C.~Villani, and B.~Wennberg.
\newblock Entropy dissipation and long-range interactions.
\newblock {\em Arch. Ration. Mech. Anal.}, 152(4):327--355, 2000.

\bibitem{AO2022}
R.~Alonso and H.~Orf.
\newblock Statistical moments and integrability properties of monatomic gas
  mixtures with long range interactions, 2022.

\bibitem{L1983}
L.~Arkeryd.
\newblock {$L^{\infty }$} estimates for the space-homogeneous {B}oltzmann
  equation.
\newblock {\em J. Statist. Phys.}, 31(2):347--361, 1983.

\bibitem{AIV2004}
A.~V. Bobylev, I.~M. Gamba, and V.~A. Panferov.
\newblock Moment inequalities and high-energy tails for {B}oltzmann equations
  with inelastic interactions.
\newblock {\em J. Statist. Phys.}, 116(5-6):1651--1682, 2004.

\bibitem{ME2016}
M.~Briant and E.~S. Daus.
\newblock The {B}oltzmann equation for a multi-species mixture close to global
  equilibrium.
\newblock {\em Arch. Ration. Mech. Anal.}, 222(3):1367--1443, 2016.

\bibitem{CT1932}
T.~Carleman.
\newblock Sur la th\'{e}orie de l'\'{e}quation int\'{e}grodiff\'{e}rentielle de
  {B}oltzmann.
\newblock {\em Acta Math.}, 60(1):91--146, 1933.

\bibitem{LC2005}
L.~Desvillettes and C.~Villani.
\newblock On the trend to global equilibrium for spatially inhomogeneous
  kinetic systems: the {B}oltzmann equation.
\newblock {\em Invent. Math.}, 159(2):245--316, 2005.

\bibitem{IV2004}
I.~M. Gamba, V.~Panferov, and C.~Villani.
\newblock On the {B}oltzmann equation for diffusively excited granular media.
\newblock {\em Comm. Math. Phys.}, 246(3):503--541, 2004.

\bibitem{IM2020}
C.~Imbert, C.~Mouhot, and L.~Silvestre.
\newblock Gaussian lower bounds for the {B}oltzmann equation without cutoff.
\newblock {\em SIAM J. Math. Anal.}, 52(3):2930--2944, 2020.

\bibitem{IS2021}
C.~Imbert and L.~E. Silvestre.
\newblock Global regularity estimates for the {B}oltzmann equation without
  cut-off.
\newblock {\em J. Amer. Math. Soc.}, 35(3):625--703, 2022.

\bibitem{SC2006}
S.~Mischler and C.~Mouhot.
\newblock Cooling process for inelastic {B}oltzmann equations for hard spheres.
  {II}. {S}elf-similar solutions and tail behavior.
\newblock {\em J. Stat. Phys.}, 124(2-4):703--746, 2006.

\bibitem{SC2009}
S.~Mischler and C.~Mouhot.
\newblock Stability, convergence to self-similarity and elastic limit for the
  {B}oltzmann equation for inelastic hard spheres.
\newblock {\em Comm. Math. Phys.}, 288(2):431--502, 2009.

\bibitem{MMR2006}
S.~Mischler, C.~Mouhot, and M.~Rodriguez~Ricard.
\newblock Cooling process for inelastic {B}oltzmann equations for hard spheres.
  {I}. {T}he {C}auchy problem.
\newblock {\em J. Stat. Phys.}, 124(2-4):655--702, 2006.

\bibitem{M2005}
C.~Mouhot.
\newblock Quantitative lower bounds for the full {B}oltzmann equation. {I}.
  {P}eriodic boundary conditions.
\newblock {\em Comm. Partial Differential Equations}, 30(4-6):881--917, 2005.

\bibitem{AB1996}
A.~Pulvirenti and B.~Wennberg.
\newblock A {M}axwellian lower bound for solutions to the {B}oltzmann equation.
\newblock {\em Comm. Math. Phys.}, 183(1):145--160, 1997.

\bibitem{V2002}
C.~Villani.
\newblock A review of mathematical topics in collisional kinetic theory.
\newblock In {\em Handbook of mathematical fluid dynamics, {V}ol. {I}}, pages
  71--305. North-Holland, Amsterdam, 2002.

\end{thebibliography}
\end{document}